\documentclass[11pt]{amsart}
\usepackage{amssymb,amsmath,amsfonts,epsfig,textcomp,amsthm}
\usepackage[margin=1in]{geometry}
\usepackage{enumerate}
\usepackage[enableskew]{youngtab}
\usepackage{graphics}
\usepackage[all,cmtip]{xy}
\usepackage{tikz}

\newtheorem{theorem}{Theorem}[section]
\newtheorem{corollary}[theorem]{Corollary}
\newtheorem{lemma}[theorem]{Lemma}
\newtheorem{proposition}[theorem]{Proposition}

\newtheorem{definition}[theorem]{Definition}
\newtheorem{example}[theorem]{Example}
\newtheorem{remark}[theorem]{Remark}

\numberwithin{equation}{section}

\DeclareMathOperator{\perm}{perm}
\DeclareMathOperator{\Int}{Int}
\DeclareMathOperator{\Des}{Des}
\DeclareMathOperator{\Comp}{Comp}
\DeclareMathOperator{\type}{type}
\DeclareMathOperator{\link}{lk}

\DeclareMathOperator{\id}{id}
\DeclareMathOperator{\Stab}{Stab}
\DeclareMathOperator{\im}{im}

\newcommand{\hz}{\hat{0}}
\newcommand{\ho}{\hat{1}}
\newcommand{\coveredby}{\prec}
\newcommand{\SSSS}{{\mathfrak S}}
\newcommand{\tensor}{\otimes}

\newcommand{\rH}{\widetilde{H}}
\newcommand{\rchi}{\widetilde{\chi}}
\newcommand{\rbeta}{\widetilde{\beta}}
\newcommand{\Ppp}{{\mathbb P}}

\newcommand{\plusdots}{+ \cdots +}
\newcommand{\timesdots}{\times \cdots \times}

\newcommand{\Pistar}{\Pi^{*}}

\newcommand{\cp}{\vec{c}\,^{\prime}}
\newcommand{\cpp}{\vec{c}\,^{\prime\prime}}
\newcommand{\ddp}{\vec{d}\,^{\prime}}

\newcommand{\onethingatopanother}[2]{\genfrac{}{}{0pt}{}{#1}{#2}}

\newcommand{\journal}[6]{{\sc #1,} #2, {\it #3} {\bf #4} (#5), #6.}

\newcommand{\book}[4]{{\sc #1,} #2, #3, #4.}
\newcommand{\thesis}[4]{{\sc #1,} ``#2,'' Doctoral dissertation, #3, #4.}

\input xy

\begin{document}

\author{Richard Ehrenborg \and Dustin Hedmark}
\title{Filters in the Partition Lattice}

\begin{abstract}
Given a filter $\Delta$
in the poset of compositions of $n$, we form
the filter $\Pistar_{\Delta}$ in the partition lattice.
We determine all the reduced homology groups
of the order complex of $\Pistar_{\Delta}$
as $\SSSS_{n-1}$-modules in terms of the reduced homology
groups of the simplicial complex $\Delta$ and in terms of Specht modules
of border shapes.
We also obtain the homotopy type of this order complex.
These results generalize work of
Calderbank--Hanlon--Robinson and Wachs
on the $d$-divisible partition lattice.
Our main theorem applies to a plethora of examples, including
filters associated to integer knapsack partitions
and 
filters generated
by all partitions having block sizes $a$ or~$b$.
We also obtain the 
reduced homology groups
of the filter generated
by all partitions having block sizes belonging
to the arithmetic progression
$a, a + d, \ldots, a + (a-1) \cdot d$,
extending work of Browdy.
\end{abstract}

\maketitle

\section{Introduction}
\label{section_introduction}

In his physics dissertation
Sylvester~\cite{Sylvester} considered 
the even partition lattice,
that is, the poset of all set partitions where the blocks have
even size.
He computed the M\"obius function of this lattice and
showed that it equals, up to a sign, the tangent number.
Stanley then introduced the $d$-divisible partition lattice. This is 
the collection of all set partitions with  blocks having size divisible by $d$, denoted by~$\Pi_{n}^d$.
He showed that the M\"obius function is, up to a sign,
the number of permutations in
the symmetric group~$\SSSS_{n-1}$
with descent set $\{d,2d, \ldots, n-d\}$;
see~\cite{exponential_structures}.

Calderbank, Hanlon and Robinson~\cite{Calderbank_Hanlon_Robinson}
continued this work by studying the top
homology group of
the order complex~$\triangle(\Pi_{n}^{d} - \{\ho\})$
and gave an explicit description of
the $\SSSS_{n-1}$-action on this homology group
in terms of a Specht module.
However, they were unable to obtain the other homology
groups and asked Wachs if it was possible that the complex
$\triangle(\Pi_{n}^{d} - \{\ho\})$ was shellable, which
would imply that the other homology groups are trivial.
Wachs~\cite{Wachs} 
proved that this was indeed the case by showing
that the poset $\Pi_{n}^{d} \cup \{\hz\}$ is
$EL$-shellable,
and thus the homotopy type
of the complex
$\triangle(\Pi_{n}^{d} - \{\ho\})$ is a wedge
of spheres of the same dimension.
Additionally, Wachs gave a different proof for the $\SSSS_{n-1}$-action
on the top homology of~$\Pi_{n}^d$,
as well as matrices for the action of~$\SSSS_{n}$ on this homology.

Ehrenborg and Jung~\cite{Ehrenborg_Jung}
further generalized the $d$-divisible partition lattice by defining
a subposet~$\Pistar_{\vec{c}}$ of the partition lattice for
a composition $\vec{c}$ of~$n$. The subposet
reduces to the $d$-divisible partition lattice when the composition $\vec{c}$
is given by $\vec{c}=(d,d,\dots,d)$.
Their work consists of three main results.
First, they showed that the M\"obius function of $\Pistar_{\vec{c}} \cup \{\hz\}$
equals, up to a given sign,
the number of 
permutations in
$\SSSS_{n}$ ending with the element $n$
having descent composition~$\vec{c}$.
Second, they showed that the order complex 
$\triangle(\Pistar_{\vec{c}} - \{\ho\})$ is 
homotopy equivalent to
a wedge of 
spheres of the same dimension. Lastly,
they proved that the action of $\SSSS_{n-1}$
on the top homology group of $\triangle(\Pistar_{\vec{c}} - \{\ho\})$ is given by
the Specht module corresponding to the composition~$\vec{c}-1$.

In the current paper we continue this research program
by considering a
more general class of filters in the partition lattice.
Let~$\Delta$ be a filter in the poset of compositions.
Since the poset of compositions is isomorphic to
a Boolean algebra, the filter~$\Delta$ under the reverse
order is a lower order ideal and hence can be viewed
as the face poset of a simplicial complex.
We define the associated filter $\Pistar_{\Delta}$ in the partition
lattice. This extends the definition of~$\Pistar_{\vec{c}}$.
In fact, when $\Delta$ is a simplex generated
by the composition~$\vec{c}$ the two definitions agree.

Our main result is that we can determine all the reduced
homology groups of the order complex
$\triangle(\Pistar_{\Delta} - \{\ho\})$
in terms of the reduced homology groups of links
in $\Delta$ and of Specht modules
of border shapes;
see Theorem~\ref{theorem_main_theorem_order_complex}.
The proof proceeds by 
showing that if the result holds
for the two complexes~$\Delta$, $\Gamma$
and also for their intersection $\Delta \cap \Gamma$,
then it holds for their union $\Delta \cup \Gamma$.
Furthermore, the proof relies on 
Mayer--Vietoris sequences
to construct the isomorphism of
Theorem~\ref{theorem_main_result}.
As our main tool, we use Quillen's fiber lemma to translate
topological data
from the filter~$Q^{*}_{\Delta}$
to the filter~$\Pistar_{\Delta}$. 

We also present a second proof of our main result,
Theorem~\ref{theorem_main_result},
using an equivariant poset fiber theorem of
Bj\"orner, Wachs and Welker~\cite{Bjorner_Wachs_Welker}.
Even though this approach is concise, it does not yield
an explicit construction of the isomorphism of
Theorem~\ref{theorem_main_result}.
In particular, our hands on approach using Mayer--Vietoris
sequences reveals how the homology groups of
$\triangle(\Pistar_{\Delta}-\{\ho\})$ are changing as
the complex $\Delta$ is built up.
Once again, the Ehrenborg--Jung
result on $Q^{*}_{\vec{c}}$
is needed
to apply the poset fiber theorem.

Our main result yields explicit expressions
for the reduced homology groups of the complex
$\triangle(\Pistar_{\Delta}-\{\ho\})$,
most notably when $\Delta$ is homeomorphic to a ball
or to a sphere.
The same holds when $\Delta$ is a shellable complex.
We are able to describe
the homotopy type of the order complex
$\triangle(\Pistar_{\Delta} - \{\ho\})$
using the homotopy fiber theorem
of~\cite{Bjorner_Wachs_Welker}.
Again, when~$\Delta$ is homeomorphic to a ball
or to a sphere, we obtain that $\Pistar_{\Delta}$ is
a wedge of spheres.
We are also able to lift discrete Morse matchings from~$\Delta$
and its links to form a discrete Morse matching
on the filter of ordered set partitions~$Q^{*}_{\Delta}$.

In Sections~\ref{section_examples}
through~\ref{section_partition_filter_a_b}
we give a plethora of examples of our results.
We consider the case when the complex~$\Delta$
is generated by a knapsack partition
to obtain a previous result of
Ehrenborg and Jung.
In Section~\ref{section_Frobenius}
we study the case when
$\Lambda$ is a semigroup of positive integers
and we consider the filter of partitions whose
block sizes belong to the semigroup~$\Lambda$.
When $\Lambda$ is generated by
the arithmetic progression
$a, a+d, a+2d, \ldots$ we are able to describe
the reduced homology groups of the associated
filter in the partition lattice.
The particular case when $d$ divides $a$
was studied by Browdy~\cite{Browdy},
where the filter $\Lambda$ consists of partitions whose block sizes
are divisible by $d$ and are greater than or equal to~$a$.
Finally, in Section~\ref{section_partition_filter_a_b}
we study the filter
%$\Pi_{n}^{\langle a,b \rangle}$
corresponding to the semigroup generated by
two relative prime integers.
%$a$ and~$b$.
Here we are able to give explicit results for
the top and bottom reduced homology groups.

Other previous work in this area is due to 
Bj\"orner and Wachs~\cite{Bjorner_Wachs_non_pure_I}.
Additionally, Sundaram
studied the subposet of the partition
lattice defined by a set of forbidden block sizes
using plethysm and the Hopf trace formula;
see~\cite{Sundaram_Hopf,Sundaram}.

We end the paper by posing questions for further study.

\section{Integer and set partitions}
\label{section_integer_set_partitions}

We define an integer partition $\lambda$ to be a finite multiset of
positive integers. Thus
the multiset
$\lambda = \{\lambda_{1}, \lambda_{2}, \ldots, \lambda_{k}\}$
is a partition of $n$ if
$\lambda_{1} + \lambda_{2} \plusdots \lambda_{k} = n$.
Sometimes it will be necessary to consider the
multiplicity of the elements of the partition~$\lambda$.
We then write
$$ \lambda
=
\{\lambda_{1}^{m_{1}}, \lambda_{2}^{m_{2}},
                              \ldots, \lambda_{p}^{m_{p}}\}
, $$
where we tacitly assume that $\lambda_{i} \neq \lambda_{j}$
for two different indices $i \neq j$.

Let $I_{n}$ be the set of all integer partitions of $n$.
We form a poset on these integer partitions where
the cover relation is given by adding two parts.
In terms of multisets the cover relation is
$$
\{\lambda_{1}, \lambda_{2},\lambda_{3}, \ldots, \lambda_{k}\}
      \coveredby
\{\lambda_{1} + \lambda_{2},\lambda_{3}, \ldots, \lambda_{k}\} . $$
Note that the partition $\{1,1, \ldots, 1\}$ is the minimal element
and $\{n\}$ is the maximal element in the partial order.

Let $\Pi_{n}$ denote the poset of all set partitions of
$[n] = \{1,2, \ldots, n\}$
where the partial order is given  by merging blocks,
that is,
$$    \{B_{1}, B_{2}, B_{3}, \ldots, B_{k}\}
     \coveredby
        \{B_{1} \cup B_{2}, B_{3}, \ldots, B_{k}\}  .  $$
The poset $\Pi_{n}$ is in fact a lattice,
called the partition lattice.
Let $|\pi|$ denote the number of blocks of the partition~$\pi$.
Furthermore, for a set partition
$\pi = \{B_{1}, B_{2}, \ldots, B_{k}\}$
define its type to be
the integer partition of $n$ given by the multiset 
$\type(\pi) = \{|B_{1}|, |B_{2}|, \ldots, |B_{k}|\}$.

The symmetric group $\SSSS_{n}$ acts on
subsets of $[n]$ by relabeling the elements.
Similarly, the symmetric group $\SSSS_{n}$ acts on
the partition lattice by relabeling the elements of the blocks.
For $\pi = \{B_{1}, B_{2}, \ldots, B_{k}\}$ a set partition
the action is given by
$\alpha \cdot \pi = \{\alpha(B_{1}), \alpha(B_{2}), \ldots, \alpha(B_{k})\}$.
Finally, when we speak about the action of the symmetric
group $\SSSS_{n-1}$, we view
the group $\SSSS_{n-1}$ as the subgroup
$\{\alpha \in \SSSS_{n} \: : \: \alpha_{n} = n\}$
of the symmetric group~$\SSSS_{n}$.

\section{Compositions and ordered set partitions}
\label{section_compositions_ordered_sets}

A composition
$\vec{c} = (c_{1}, c_{2}, \ldots,c_{k})$ of $n$
is an ordered list of positive integers such that
$c_{1} + c_{2} \plusdots c_{k} = n$.
Let $\Comp(n)$ be the set of all compositions of $n$.
We make $\Comp(n)$ into a poset by
introducing the cover relation
given by adding adjacent entries, that is,
$$  (c_{1},\ldots,c_{i},c_{i+1},\ldots,c_{k})
   \coveredby
      (c_{1},\ldots,c_{i}+c_{i+1},\ldots,c_{k}) . $$
The poset $\Comp(n)$ is isomorphic to 
the Boolean algebra on $n-1$ elements.
Note that 
$(1,1, \ldots, 1)$ and $(n)$
are the minimal and maximal elements of $\Comp(n)$, respectively.
Define the type of a composition
$\vec{c} = (c_{1}, c_{2}, \ldots, c_{k})$
to be the integer partition
$\type(\vec{c}\,) = \{c_{1}, c_{2}, \ldots,c_{k}\}$ of~$n$.
Furthermore, let $|\vec{c}\,|$ denote the number of parts
of the composition~$\vec{c}$. 

For a composition
$\vec{c} = (c_{1}, c_{2}, \ldots, c_{k})$
of $n$, the {\em multinomial coefficient}
is given by
$$  \binom{n}{\vec{c}}
    =
\binom{n}{c_{1}, c_{2}, \ldots, c_{k}}
     =
\frac{n!}{c_{1}! \cdot c_{2}! \cdots c_{k}!}  .  $$

For $\alpha\in\SSSS_{n}$, let the descent set of $\alpha$,
denoted by $\Des(\alpha)$, be the subset of $[n-1]$
given by $\Des(\alpha)=\{i\in[n-1]\: : \: \alpha(i)>\alpha(i+1)\}$.
Throughout this paper it will be more
convenient to consider $\Des(\alpha)$ as a composition of $n$,
namely,
if $\Des(\alpha)=\{i_{1} < i_{2} < \cdots < i_{k}\}$,
then we consider $\Des(\alpha)$ as a composition of $n$
given by $\Des(\alpha)=(i_{1},i_{2}-i_{1}, \ldots, i_{k}-i_{k-1},n-i_{k})$.
Note that the identity permutation
$(1,2,\ldots,n)$ has descent composition~$(n)$.

Let $\beta_{n}(\vec{c}\,)$ be the number of permutations
$\alpha$ in $\SSSS_{n}$ such that $\Des(\alpha)=\vec{c}$.
Likewise, define~$\beta_{n}^{*}(\vec{c}\,)$
to be the number of permutations $\alpha$ in $\SSSS_{n}$ with 
descent composition $\vec{c}$ and $\alpha(n)=n$.
Observe that
\begin{equation}
\binom{n-1}{c_{1}, \ldots, c_{k-1}, c_{k}-1}
=
\sum_{\onethingatopanother
                 {\vec{d} \in \Comp(n)}
                 {\vec{c}\,\leq\vec{d}}}
\beta_{n}^{*}(\vec{d}\,)
.
\label{equation_beta_star_binomial}
\end{equation}

An {\em ordered set partition} $\sigma = (C_{1}, C_{2}, \ldots, C_{p})$
of~$[n]$ is a 
list of non-empty blocks such that
the set 
$\{C_{1}, C_{2}, \ldots, C_{p}\}$ is a partition of the set $[n]$,
where the order of the blocks now matters.
Let $|\sigma|$ denote the number of blocks
in the ordered set partition~$\sigma$.

Let $Q_{n}$ be the set of all ordered set partitions on the set $[n]$.
Introduce a partial order on~$Q_{n}$ where the cover relation
is joining adjacent blocks, that is,
$$
(C_{1}, \ldots, C_{i}, C_{i+1}, \ldots, C_{p})
\coveredby
(C_{1}, \ldots, C_{i} \cup C_{i+1}, \ldots, C_{p}) .
$$
Observe that the poset~$Q_{n}$ has the maximal element~$([n])$,
along with $n!$ minimal elements,
namely the ordered set partitions
$(\{\alpha_{1}\}, \{\alpha_{2}\}, \ldots, \{\alpha_{n}\})$, one for each
permutation $\alpha_{1} \alpha_{2} \cdots \alpha_{n} \in \SSSS_{n}$.
Moreover, every interval in~$Q_{n}$
is isomorphic to a Boolean algebra. 

Define the {\em type} of an ordered set partition
$\sigma = (C_{1}, C_{2}, \ldots, C_{k})$
to be the composition of~$n$ given
by $\type(\sigma) = (|C_{1}|, |C_{2}|, \ldots, |C_{k}|)$.
\begin{definition}
\label{definition_sigma}
For a permutation $\alpha \in \SSSS_{n}$
and a composition
$\vec{d} = (d_{1}, d_{2}, \ldots, d_{k})$ of~$n$,
let~$\sigma(\alpha,\vec{d}\,)$ denote 
the unique ordered set partition in $Q_{n}$
of type~$\vec{d}$ whose elements are given,
in order, by the permutation~$\alpha$,
that is,
$$
\sigma(\alpha,\vec{d}\,)
  =
(\{\alpha(1), \ldots, \alpha(d_{1})\},
\{\alpha(d_{1}+1), \ldots, \alpha(d_{2})\},
\ldots,
\{\alpha(d_{k-1}+1), \ldots, \alpha(n)\})
.
$$
\end{definition}

Finally, the symmetric group $\SSSS_{n}$ acts on
ordered set partitions by relabeling, that is
$$  \alpha \cdot (C_{1}, C_{2}, \ldots, C_{k})
    =
     (\alpha(C_{1}), \alpha(C_{2}), \ldots, \alpha(C_{1}))  .  $$

\section{Topological considerations}
\label{section_topological_considerations}

Let $P$ be a poset.
Recall the order complex of $P$, denoted~$\triangle(P)$,
is the simplicial
complex whose $i$-dimensional faces
are the chains in $P$ with $i+1$ elements.
If $P$ has a minimal element~$\hz$
or a maximal element~$\ho$,
then $\triangle(P)$ is a contractible complex.
Thus we will be removing these elements to
ensure interesting topology.

Recall a simplicial complex~$\Delta$
is a finite collection of sets such that
the empty set belongs to~$\Delta$
and
$\Delta$~is closed under inclusion.
We will find it easier to view a simplicial complex
as a partially ordered set $\Delta$ such that 
(i) $\Delta$ has a unique minimal element~$\hz$
and
(ii) every interval $[\hz,x]$ for $x \in \Delta$ is 
isomorphic to a Boolean algebra.
A poset $P$ satisfying these conditions is called
a \emph{simplicial poset}.
Notice that a poset $P$ is simplicial if $P$ is the face poset
of a simplicial complex. 
Furthermore, note that the second condition in the 
definition of a simplicial poset makes the poset~$\Delta$ ranked
since every saturated chain between the minimal element~$\hz$
and an element~$x$ has the same length.
Thus the {\em dimension} of an element $x$ is defined
by its rank minus one, that is,
$\dim(x) = \rho(x) - 1$.

A filter in a poset $P$ is an upper order ideal.
Hence if $F$ is a filter in $P$, 
then the dual filter~$F^{*}$  in the dual poset~$P^{*}$
is now a lower order ideal.
In particular,
if $\Delta \subseteq \Comp(n)$ is a filter, 
since upper order ideals in~$\Comp(n)$
are isomorphic to Boolean algebras,
the dual of $\Delta$ is a simplicial poset
in the dual space~$\Comp(n)^{*}$, which 
has cover relation given by splitting rather than merging.
To emphasize that we have dualized, we use~$\leq^{*}$
to denote the order relation in the dualized~$\Comp(n)$.

Lastly, the {\em link} of a face $F$ in
a simplicial complex $\Delta$ is given by
$\link_{F}(\Delta) = \{G \in \Delta \: : \: F \cup G \in \Delta, \:
F \cap G = \emptyset\}$.
However, working with the poset definition of a simplicial
complex, we have the following equivalent definition of the link.
The link is the principle filter generated by the face~$x$,
that is,
$\link_{x}(\Delta) = \{y \in \Delta \: : \: x \leq y\}$.
One advantage of this definition is that we do not have
to relabel the faces when considering the link.

From now on our simplicial complex~$\Delta$
will be a filter in the composition lattice~$\Comp(n)$, 
with the dual order~$\leq^{*}$.

Let $C_{k}(\Comp(n))$ be the linear span over~$\mathbb{C}$ of
all compositions of~$n$ into $k+2$ parts.
We obtain a chain complex by defining the
boundary map as follows.
Define the map
$\partial_{k,j}:C_{k}(\Comp(n)) \longrightarrow C_{k-1}(\Comp(n))$ by
$$ \partial_{k,j}(c_{1},\ldots,c_{j},c_{j+1},\ldots,c_{k+2})
   =
(c_{1},\ldots,c_{j}+c_{j+1},\ldots,c_{k+2})  . $$
Then the boundary map on $\Comp(n)$ is given by
$\partial_{k} = \sum_{j=1}^{k+1} (-1)^{j-1} \cdot \partial_{k,j}$

Consider the dual order on the set of ordered set partitions~$Q_{n}$.
For $\Delta \subseteq \Comp(n)$ a complex, let
$Q_{\Delta}=\{\tau \in Q_{n} \,:\, \type(\tau)\in\Delta\}$.
The filter $Q_{\Delta}$
is also a simplicial poset, so we refer to~$Q_{\Delta}$ as a complex. 

Define $C_{k}(Q_{n})$ to be the linear span over~$\mathbb{C}$ of all ordered
set partitions of $[n]$ with $k+2$ blocks.
The boundary map
$\partial_{k} : C_{k}(Q_{n}) \longrightarrow C_{k-1}(Q_{n})$
on $Q_{n}$ is given by
$\partial_{k}(\sigma(\alpha,\vec{d}\,))
=
\sigma(\alpha,\partial_{k}(\vec{d}\,))$,
where $\partial_{k}(\vec{d}\,)$
is the boundary map applied to the composition $\vec{d}$
in $C_{k}(\Comp(n))$, and 
where $\sigma(\alpha,\vec{c}\,)$
is given in Definition~\ref{definition_sigma}.
This boundary map is inherited by the subcomplex~$Q_{\Delta}$.

Finally, for simplicial complexes $\Delta$ and~$\Gamma$
in $\Comp(n)$ and $\Comp(m)$ respectively, their {\em join}
is defined to be poset
$$ \Delta * \Gamma
  =
     \{\vec{c} \circ \vec{d} 
           \: : \:
       \vec{c} \in \Delta, \vec{d} \in \Gamma \} ,  $$
where $\circ$ denote the concatenation of compositions.
Note that the join $\Delta * \Gamma$ has
the composition~$(n,m)$ as its minimal element.
Furthermore, we have the following basic lemma
on Morse matchings of joins of complexes.
\begin{lemma}
\label{lemma_critical_cell_join}
Let $\Delta$ and $\Gamma$ be two complexes
in $\Comp(m)$ and $\Comp(m)$ respectively,
each having a discrete Morse matching.
Let~$\Delta^{c}$ and 
$\Gamma^{c}$ be the sets of critical cells of $\Delta$ and~$\Gamma$, respectively.
Then the join~$\Delta * \Gamma$ has a Morse
matching where the critical cells are
$$
     \{\vec{c} \circ \vec{d} 
           \: : \:
       \vec{c} \in \Delta^{c}, \vec{d} \in \Gamma^{c} \} .  $$
\label{lemma_join_critical_cells}
\end{lemma}
\begin{proof}
Define a matching of the join $\Delta * \Gamma$ as follows.
If $\vec{c} \coveredby \cp$ is an edge in the discrete Morse
matching of~$\Delta$ and
$\vec{d} \in \Gamma$
then match
$\vec{c} \circ \vec{d} \coveredby \cp \circ \vec{d}$.
If $\vec{c}$ is a critical cell of $\Delta$
and
$\vec{d} \coveredby \ddp$ is an edge in the discrete Morse
matching of $\Gamma$
then match
$\vec{c} \circ \vec{d} \coveredby \vec{c} \circ \ddp$.
It is straightforward to verify that this matching is acyclic
and that the set of critical cells is as described.
\end{proof}

\section{Border strips and Specht modules}
\label{section_border_strips}

A border strip $B$ is a connected skew-shape which does not contain
a two by two square.
For each composition $\vec{c} = (c_{1}, c_{2}, \ldots, c_{k})$
there is a unique border strip
such that the number of boxes in the $i$th row is given
by $c_{i}$ and every two adjacent rows overlap in 
one position.
Denote this border strip by~$B(\vec{c}\,)$. 
See Figure~\ref{figure_stabilizer} for an example.

In an analogous fashion, for a composition $\vec{c} \in \Comp(n)$
we define the border shape~$A(\vec{c}\,)$ to be the skew-shape
whose $i$th row has length~$c_i$ such that
the rows of $A(\vec{c}\,)$ are non-overlapping.

Let $R_{i}$ be the interval
$[c_{1} + \cdots + c_{i-1} + 1, c_{1} + \cdots + c_{i-1} + c_{i}]$.
The row stabilizer of the border strip $B(\vec{c}\,)$
is the subgroup
$\SSSS_{R_{1}} \times \SSSS_{R_{2}} \timesdots \SSSS_{R_{k}}$
of the symmetric group~$\SSSS_{n}$.

Since the poset $\Comp(n)$ of all compositions of $n$ is
a isomorphic to Boolean algebra,
every composition has a complementary
composition $\vec{c}^{\,c}$.
To obtain the complement of composition
write every part of the composition as a sum of $1$s 
where we separate the parts with commas.
Then the complement is obtained by
exchanging the plus signs and the commas.
Similarly, the column stabilizer is defined as
the row stabilizer of the border strip of the complementary
composition. More precisely, let
$(d_{1}, d_{2}, \ldots, d_{p})$
be the complementary composition~$\vec{c}\,^{c}$
and let
$K_{i}$ be the interval
$[d_{1} \plusdots d_{i-1} + 1,  d_{1} \plusdots d_{i-1} + d_{i}]$.
Then the column stabilizer is the subgroup
$\SSSS_{\vec{c}}^{C}
=
\SSSS_{K_{1}} \times \SSSS_{K_{2}} \timesdots \SSSS_{K_{p}}$.
See Figure~\ref{figure_stabilizer}.

We now review some basic representation theory of the symmetric group. For a less terse introduction,
see~\cite[Chapter~3]{Sagan}.
A border strip tableau~$t$ of shape~$\vec{c}$
is a filling of the border strip~$B(\vec{c}\,)$.
We say a tableau $t$ is {\em standard} if the entries of~$t$
are increasing along the rows from left to right
and increasing down the columns.
A border strip tabloid, denoted~$[t]$,
is a border strip tableau under row equivalence.
Define the permutation module, $M^{B(\vec{c}\,)}$,
to be the vector space with basis elements given
by all tabloids of shape~$B(\vec{c}\,)$.
A polytabloid is defined
by the alternating sum
$e_{t}
=
\sum_{\gamma \in \SSSS^{C}_{\vec{c}}}
       (-1)^{\gamma} \cdot [\gamma \cdot t]$,
where $\SSSS^{C}_{\vec{c}}$
is the column stabilizer of the tableau~$t$ of shape~$\vec{c}$.
Lastly, the Specht module,
denoted $S^{B(\vec{c}\,)}$, is the subspace of $M^{B(\vec{c}\,)}$ generated by polytabloids.
The dimension of
the Specht module~$S^{B(\vec{c}\,)}$
is given by the descent set statistics~$\beta_{n}(\vec{c}\,)$,
while the dimension of the permutation module~$M^{B(\vec{c}\,)}$
is given by the multinomial coefficient~$\binom{n}{\vec{c}}$.

We now define two operations on compositions.
The motivation comes from the associated
Specht and permutation modules.
For a composition
$\vec{c} = (c_{1},\ldots,c_{k-1},c_{k})$
let $\vec{c}-1$ denote the composition
$(c_{1},\ldots,c_{k-1},c_{k}-1)$
if $c_{k} \geq 2$,
and otherwise let $\vec{c}-1$ denote the empty composition.
Similarly,
let $\vec{c}/1$ denote the composition
$(c_{1},\ldots,c_{k-1},c_{k}-1)$
if $c_{k} \geq 2$,
and otherwise let $\vec{c}/1$ denote
the composition $(c_{1},\ldots,c_{k-1})$.
Note that
if $\vec{c}$ is a composition of~$n$
then
$\vec{c}/1$ is always a composition of~$n-1$.

For a composition $\vec{c}$ of $n$
let $B^{*}(\vec{c}\,)$ denote the border strip~$B(\vec{c} - 1)$.
All our results of this paper
are stated in terms of
the Specht modules $S^{B^{*}(\vec{c}\,)}$
where the group action is by~$\SSSS_{n-1}$.
We think of this Specht module as a submodule
of $S^{B(\vec{c}\,)}$ spanned by all standard Young tableaux
where the northeastern-most box is filled with~$n$.
Note that when the composition ends with the entry~$1$,
there are no such standard Young tableaux, and hence
$S^{B^{*}(\vec{c}\,)}$ is the zero module.

For a composition~$\vec{c}$
define the two shapes
$B^{\#}(\vec{c}\,)=B(\vec{c}/1)$
and
$A^{\#}(\vec{c}\,)=A(\vec{c}/1)$.
Observe that the permutation module~$M^{B^{\#}(\vec{c}\,)}$
is a submodule of~$M^{B(\vec{c}\,)}$
That is, the span of the tabloids in~$M^{B(\vec{c}\,)}$
where the tabloids has the filling~$n$ in the northeastern-most box
is the module~$M^{B^{\#}(\vec{c}\,)}$.

Furthermore,
the dimensions of
the Specht module~$S^{B^{*}(\vec{c}\,)}$
and
the permutation module~$M^{B^{\#}(\vec{c}\,)}$
are~$\beta^{*}_{n}(\vec{c}\,)$ and~$\binom{n}{\vec{c}/1}$, respectively.
Additionally, we have the decomposition
$$
M^{B^{\#}(\vec{c}\,)}
\cong_{\SSSS_{n-1}}
\bigoplus_{\vec{c} \leq \vec{d}} S^{B^{*}(\vec{d}\,)} ,
$$
which is the representation theoretic analogue of
equation~\eqref{equation_beta_star_binomial}.
See Lemma~\ref{lemma_permutation_module_isomorphism}
for a proof.

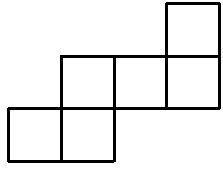
\begin{figure}[t]
\setlength{\unitlength}{0.7mm}
\begin{center}
\begin{picture}(40,40)(0,0)
\thicklines
\put(0,0){\line(1,0){20}}
\put(0,10){\line(1,0){40}}
\put(10,20){\line(1,0){30}}
\put(30,30){\line(1,0){10}}
\put(0,0){\line(0,1){10}}
\put(10,0){\line(0,1){20}}
\put(20,20){\line(0,-1){20}}
\put(30,10){\line(0,1){20}}
\put(40,10){\line(0,1){20}}
\end{picture}
\end{center}
\caption{The border strip $B(\vec{c}\,)$ associated with
the composition $\vec{c} = (2,3,1)$.
The row stabilizer is the group
$\SSSS_{\vec{c}} =
\SSSS_{[1,2]} \times \SSSS_{[3,5]} \times \SSSS_{[6,6]}$.
Note that $\vec{c} = (1+1,1+1+1,1)$
so that the complementary composition is
$\vec{c}^{\,c} = (1,1+1,1,1+1) = (1,2,1,2)$.
Hence the column stabilizer is
$\SSSS_{\vec{c}}^{C} =
\SSSS_{[1,1]} \times \SSSS_{[2,3]} \times \SSSS_{[4,4]}
    \times \SSSS_{[5,6]}$.}
\label{figure_stabilizer}
\end{figure}

\section{The ordered partition filter $Q_{\Delta}^{*}$}
\label{section_the_main_result}

We now introduce the ordered partition filter~$Q_{\Delta}^{*}$.
This filter will serve us as an important stepping stone
to understanding the topology of 
general filters in the partition lattice.
The transition from~$Q_{\Delta}^*$ to the partition lattice
uses Quillen's Fiber Lemma; see
Section~\ref{section_filters_in_the_set_partition_lattice}.
Note that by considering the reverse orders
in $\Comp(n)$ and in $Q_{n}$ we obtain two simplicial
posets.
Hence for~$\Delta$ a non-empty filter in~$\Comp(n)$, 
we view~$\Delta$ as a simplicial complex under
the reverse order~$\leq^{*}$. See the discussion in
Section~\ref{section_topological_considerations}.

\begin{definition}
Let $\Delta$ be a filter in $\Comp(n)$, that is,
$\Delta$ is a simplicial complex consisting of compositions of~$n$.
Define the ordered partition filter~$Q^{*}_{\Delta}$ to be all
ordered set partitions whose type is in
the complex~$\Delta$ and whose last block contains the
element $n$, that is,
$$
Q^{*}_{\Delta}
= \{\sigma = (C_{1}, C_{2}, \ldots, C_{k}) \in Q_{n}
      \: : \: \type(\sigma) \in \Delta, \:
             n \in C_{k}\} . $$
\end{definition}
Note that we view $Q^{*}_{\Delta}$
as a simplicial complex. Our purpose is to study the reduced homology groups
of this complex.

Recall that the link of a composition $\vec{c}$
in~$\Delta$ is the filter
$$ \link_{\vec{c}}(\Delta)
    =
     \{\vec{d} \in \Delta \: : \: \vec{d}\leq^{*}\vec{c}\,\} ,  $$
where $\leq^{*}$ is the reverse of
the partial order of $\Comp(n)$.
Since $\link_{\vec{c}}(\Delta)$ is now a simplicial poset with 
minimal element~$\vec{c}$,
we have a dimension shift from $\Delta$ to $\link_{\vec{c}}(\Delta)$
given by
\begin{equation}
\dim_{\link_{\vec{c}}(\Delta)}(\vec{d}\,)
=
\dim_{\Delta}(\vec{d}\,)-|\vec{c}\,|+1
\label{equation_dimension}
\end{equation}
for $\vec{d}\in\link_{\vec{c}\,}(\Delta)$.
\begin{remark}
{\rm
The symmetric group $\SSSS_{n-1}$ acts on $Q_{\Delta}^{*}$
by permutation, whereas
the action of~$\SSSS_{n-1}$ on the complex $\Delta$
is the trivial action.
Furthermore, the type map from
$Q_{\Delta}^{*}$ to $\Delta$ respects this action,
since the two ordered set partitions
$\sigma$ and $\tau \cdot \sigma$
have the same type.
}
\label{remark_action}
\end{remark}
A special case of $Q^{*}_{\Delta}$ is when
the simplicial complex $\Delta$ is a simplex,
that is, $\Delta$ is generated by one composition~$\vec{c}$. 
This case was studied by Ehrenborg and Jung
in~\cite{Ehrenborg_Jung}. 
Their results are given below.
\begin{theorem}[Ehrenborg--Jung]
Let $\vec{c}$ be a composition of $n$ into $k$ parts.
Then the complex~$Q^{*}_{\vec{c}}$ is a wedge
of $\beta^{*}_{n}(\vec{c}\,)$ spheres of dimension~$k-2$.
Furthermore,
the top homology group~$\rH_{k-2}(Q^{*}_{\vec{c}})$
is isomorphic to
the Specht module~$S^{B^{*}(\vec{c}\,)}$
as an $\SSSS_{n-1}$-module.
This isomorphism 
$\phi : S^{B^{*}(\vec{c}\,)} \longrightarrow \rH_{k-2}(Q^{*}_{\vec{c}})$
is given by 
$$ \phi\left(e_{t}\right)
  =
\sum_{\gamma \in \SSSS^{C}_{\vec{c}}}
(-1)^{\gamma} \cdot \sigma(\alpha\cdot\gamma,\vec{c}\,) ,
$$
where the permutation $\alpha \in \SSSS_{n}$
is obtained by reading the entries of
the tabloid~$t$
from southwest to northeast
and attaching the element~$n$ at the end.
\label{theorem_Ehrenborg_Jung}
\end{theorem}
Note that Ehrenborg and Jung formulated their result
in terms of pointed set partitions. That is, our
notation~$Q^{*}_{\vec{c}}$
is~$\Delta_{\vec{d}}$ in their notation, where
$\vec{d} = (c_{1}, \ldots, c_{k-1},c_{k}-1)$.
They allow the last entry of a composition to be zero
and similarly the last entry of an ordered set partition
to be empty.
Moreover, our notation~$\Pistar_{\vec{c}}$
is in their notation~$\Pi^{\bullet}_{\vec{d}}$.

We can now state the main result of this section.
\begin{theorem}
Let $\Delta$ be a simplicial complex of compositions of $n$.
Then the $i$th reduced homology group
of the simplicial complex~$Q^{*}_{\Delta}$
is given by
$$
\rH_{i}(Q^{*}_{\Delta})
\cong
\bigoplus_{\vec{c}\, \in \Delta}
       \rH_{i-|\vec{c}\,|+1}(\link_{\vec{c}\,}(\Delta))
   \otimes
          S^{B^{*}(\vec{c}\,)}.    $$
Furthermore, this isomorphism holds as
$\SSSS_{n-1}$-modules.
\label{theorem_main_result}
\end{theorem}
We will prove Theorem~\ref{theorem_main_result}
in Sections~\ref{section_homomorphism_phi}
through~\ref{section_the_building_step}.

\section{The homomorphism $\phi^{\Delta}_{i}$}
\label{section_homomorphism_phi}

In this section and the next two sections we present
a proof of Theorem~\ref{theorem_main_result}.
The major step
is to show that if Theorem~\ref{theorem_main_result}
holds for $\Delta$, $\Gamma$,
and the intersection $\Delta \cap \Gamma$,
then it also holds for
the union~$\Delta \cup \Gamma$.
This step requires Mayer--Vietoris sequences.
When
$\Delta$ is generated by
a single composition~$\vec{c}$ in~$\Comp(n)$,
the result follows from 
Theorem~\ref{theorem_Ehrenborg_Jung}.
Finally, since any simplicial complex is a union of 
simplices, Theorem~\ref{theorem_main_result}
will hold for arbitrary simplicial complexes~$\Delta$ in~$\Comp(n)$.

We begin by defining the isomorphism of
Theorem~\ref{theorem_main_result} explicitly.
Throughout the paper we will let $i_{\vec{c}}$
denote the shift $i-|\vec{c}\,|+1$.
\begin{definition}
\label{definition_D_chain_complex}
Let $D_{i}^{\vec{c}\,}(\Delta)$ be the tensor product
$C_{i_{\vec{c}\,}}(\link_{\vec{c}\,}(\Delta))\otimes M^{B^{\#}(\vec{c}\,)}$ 
where
$C_{j}(\link_{\vec{c}\,}(\Delta))$ is the $j$th chain group of the link
$\link_{\vec{c}\,}(\Delta)$.
Let $D^{\vec{c}\,}(\Delta)$
be the chain complex whose $i$th chain group
is $D^{\vec{c}\,}_{i}(\Delta)$ and whose boundary map is
$\partial \otimes \id$.
Lastly, let $D(\Delta)$ be the chain complex with $i$th chain group
$\bigoplus_{\vec{c}\,\in\Delta}D^{\vec{c}}_{i}(\Delta)$
with the differential
$\bigoplus_{\vec{c}\in\Delta} \partial \otimes \id$.
\end{definition}
\begin{definition}
\label{definition_E_chain_complex}
Define the chain complex $E^{\vec{c}}(\Delta)$
analogous to $D^{\vec{c}}(\Delta)$ of
Definition~\ref{definition_D_chain_complex} above by
replacing the permutation module $M^{B^{\#}(\vec{c}\,)}$
with the Specht module~$S^{B^{*}(\vec{c}\,)}$. 
We also have the corresponding chain complex~$E(\Delta)$
with the same differential.
\end{definition}

\begin{lemma}
\label{lemma_chain_homology}
The homology of the chain complexes
$D(\Delta)$ and $E(\Delta)$ are given by 
\begin{align*}
\rH_{i}(D(\Delta))
& \cong_{\SSSS_{n-1}}
\bigoplus_{\vec{c}\in\Delta}
\rH_{i_{\vec{c}}\,}(\link_{\vec{c}\,}(\Delta))\otimes M^{B^{\#}(\vec{c}\,)} , \\ 
\rH_{i}(E(\Delta))
& \cong_{\SSSS_{n-1}}
\bigoplus_{\vec{c}\in\Delta}
\rH_{i_{\vec{c}}\,}(\link_{\vec{c}\,}(\Delta))\otimes S^{B^{*}(\vec{c}\,)} . 
\end{align*}
\end{lemma}
\begin{proof}
The homology of the chain complex $D^{\vec{c}}(\Delta)$
is given by
$\ker(\partial_{i_{\vec{c}} } \tensor \id)/\im(\partial_{i_{\vec{c}} +1}
      \tensor \id)
\cong
(\ker(\partial_{i_{\vec{c}} }) \tensor M^{B^{\#}(\vec{c}\,)})/
(\im(\partial_{i_{\vec{c}} +1}) \tensor M^{B^{\#}(\vec{c}\,)})
\cong
\ker(\partial_{i_{\vec{c}} })/\im(\partial_{i_{\vec{c}} +1})
               \tensor M^{B^{\#}(\vec{c}\,)}
\cong
\rH_{i_{\vec{c}}}(\link_{\vec{c}}(\Delta)) \tensor M^{B^{\#}(\vec{c}\,)}$.
The analogous result holds for $E^{\vec{c}}(\Delta)$
and the lemma follows by taking direct sums.
\end{proof}

For the rest of this section we let $t$ denote a tabloid
in the permutation module $M^{B^{\#}(\vec{c}\,)}$
and $\alpha \in \SSSS_{n}$
is the permutation obtained by reading the entries of
the tabloid $t$ in increasing order from
southwest to northeast and adjoining the element $n$
at the end.

\begin{definition}
\label{definition_group_action}
The $\SSSS_{n-1}$-action on $D^{\vec{c}}_{i}(\Delta)$
is given by
$\tau \cdot (\vec{d} \otimes t)=\vec{d} \otimes (\tau \circ t)$,
for $\tau \in \SSSS_{n-1}$ and $\vec{d} \otimes t$
a basis element of
$D^{\vec{c}}_{i}(\Delta)=C_{i_{\vec{c}}}(\link_{\vec{c}}(\Delta)) \otimes M^{B^{\#}(\vec{c}\,)}$. 
\end{definition}
Notice that Definition~\ref{definition_group_action}
states that $\SSSS_{n-1}$ acts on $D^{\vec{c}}_{i}(\Delta)$
by acting trivially on the chain group
$C_{i_{\vec{c}}}(\link_{\vec{c}}(\Delta))$
and by relabeling on $M^{B^{\#}(\vec{c}\,)}$.

\begin{definition}
For a simplicial complex $\Delta$ and
a composition $\vec{c}$ in $\Delta$
define the map
$$    \phi^{\Delta,\vec{c}}_{i}
   : C_{i_{\vec{c}}}(\link_{\vec{c}}(\Delta)) \tensor M^{B^{\#}(\vec{c}\,)}
     \longrightarrow
     C_{i}(Q^{*}_{\Delta})   , $$
on basis elements by
$\phi_{i}^{\Delta,\vec{c}}(\vec{d} \otimes t)=\sigma(\alpha,\vec{d}\,)$.
\end{definition}
Since $\vec{d}\in C_{i_{\vec{c}}}(\link_{\vec{c}}(\Delta))$ is a basis element, 
we know that $\vec{d}$ is a simplex of $\link_{\vec{c}}(\Delta)$ of dimension $i_{\vec{c}}=i-|\vec{c}\,|+1$,
and thus by the dimension shift in
equation~\eqref{equation_dimension},
we have that $|\vec{d}\,|=i+2$, so that
$\phi_{i}^{\Delta,\vec{c}}(\vec{d}\,)=\sigma(\pi,\vec{d}\,)$
is an ordered partition of dimension~$i$. Lastly,
since tabloids in~$M^{B^{\#}(\vec{c}\,)}$ have $n$ in the last block,
we are guaranteed that
$\phi_{i}^{\Delta,\vec{c}}(\vec{d}\,)\in C_{i}(Q_{\Delta}^{*})$.
\begin{lemma}
\label{lemma_equivariance}
The map
$\phi_{i}^{\Delta,\vec{c}}:
    C_{i_{\vec{c}}}(\link_{\vec{c}}(\Delta)) \tensor M^{B^{\#}(\vec{c}\,)}
     \longrightarrow
     C_{i}(Q^{*}_{\Delta})$
respects the $\SSSS_{n-1}$-action.
\end{lemma}
\begin{proof}
Let $\tau \in \SSSS_{n-1}$ and $\vec{d} \otimes t$
be a basis element of
$C_{i_{\vec{c}}}(\link_{\vec{c}}(\Delta)) \tensor M^{B^{\#}(\vec{c}\,)}$.
Then we have
\begin{align*}
\phi_{i}^{\Delta,\vec{c}}(\tau\cdot(\vec{d}\otimes t))
&
= \phi_{i}^{\Delta,\vec{c}}(\vec{d}\otimes (\tau \cdot t))
= \sigma(\tau \cdot \alpha,\vec{d}\,)
= \tau \cdot \sigma(\alpha,\vec{d}\,)
= \tau \cdot \phi_{i}^{\Delta,\vec{c}}(\vec{d} \otimes t).
\qedhere 
\end{align*}
\end{proof}
\begin{lemma}
\label{lemma_commute_diagram_1}
The map $\phi^{\Delta,\vec{c}}_{i}$ is an equivariant chain map 
between the complexes
$D^{\vec{c}\,}(\Delta)$ and~$C_{i}(Q^{*}_{\Delta})$. That is, the following diagram commutes:
$$
\xymatrixcolsep{5pc}
\xymatrix{
D_{i}^{\vec{c}\,}(\Delta)
\ar[d]^{\phi^{\Delta,\vec{c}}_{i}} \ar[r]^{\partial \tensor \id} &
D_{i-1}^{\vec{c}\,}(\Delta)
\ar[d]^{\phi^{\Delta,\vec{c}}_{i-1}} \\
C_{i}(Q^{*}_{\Delta})
\ar[r]^{\partial} &
C_{i-1}(Q^{*}_{\Delta})}  
$$
\end{lemma}
\begin{proof}
The boundary map~$\partial$ of~$\Comp(n)$
as well as the boundary map~$\partial$ of~$Q_{\Delta}^{*}$
are given in
Section~\ref{section_topological_considerations}.
Let 
$\vec{d} \otimes t \in
     C_{i_{\vec{c}}}(\link_{\vec{c}}(\Delta)) \tensor M^{B^{\#}(\vec{c}\,)}$.
Tracing first right then down we obtain:
$$\phi_{i-1}^{\Delta,\vec{c}} \circ (\partial \otimes \id)(\vec{d} \otimes t)
 =
\phi_{i-1}^{\Delta,\vec{c}}(\partial(\vec{d}\,) \otimes t)
 =
\sigma(\alpha,\partial(\vec{d}\,)).$$
Next, we trace down then right to obtain the same result:
$$
\partial\circ\phi_{i}^{\Delta,\vec{c}}(\vec{d} \otimes t)
=
\partial(\sigma(\alpha,\vec{d}\,))
=
\sigma(\alpha,\partial(\vec{d}\,)).$$
The equivariance of $\phi_{i}^{\Delta,\vec{c}}$ is a consequence of Lemma~\ref{lemma_equivariance}.
\end{proof}

\begin{lemma}
The map $\phi_{i}^{\Delta, \vec{c}}$
induces a map
$$\phi_{i}^{\Delta, \vec{c}}:
\rH_{i_{\vec{c}}}(\link_{\vec{c}\,}(\Delta))\otimes M^{B^{\#}(\vec{c}\,)}
    \longrightarrow
\rH_{i}(Q^{*}_{\Delta}) $$
given by
$\phi_{i}^{\Delta}(\overline{\vec{d}\,} \otimes t)
=
\overline{\sigma(\alpha,\vec{d}\,)}$,
for $\vec{d}\in C_{i_{\vec{c}}}(\link_{\vec{c}\,}(\Delta))$ a cycle.
\label{lemma_passing_to_homology}
\end{lemma}
\begin{proof}
Since $\phi_{i}^{\Delta, \vec{c}}$ is 
an equivariant chain map 
between the chain complexes
$D^{\vec{c}\,}(\Delta)$
and~$C_{i}(Q^{*}_{\Delta})$
by
Lemma~\ref{lemma_commute_diagram_1},
the result follows.
\end{proof}
For the rest of the paper, the use of the bar to indicate the quotient
in passing from the chain space to the homology group
will be suppressed for ease of notation.

\begin{definition}
Define the map
$\phi^{\Delta}_{i}$
from 
$D_{i}(\Delta)
=
\bigoplus_{\vec{c}\,\in\Delta} D^{\vec{c}}_{i_{\vec{c}}}(\Delta)$
to $C_{i}(Q^{*}_{\Delta})$
by adding
all the 
$\phi^{\Delta,\vec{c}}_{i}$ maps together,
that is,
\begin{equation}
\label{equation_phi}
\phi^{\Delta}_{i}
=
\sum_{\vec{c} \in \Delta}
\phi^{\Delta,\vec{c}}_{i} . 
\end{equation}
\end{definition}
Observe that $\phi^{\Delta}_{i}$ restricts
to a map from 
$E_{i}(\Delta)$
to $C_{i}(Q^{*}_{\Delta})$.
Therefore $\phi_{i}^{\Delta}$ also induces  a map
from
$\rH_{i}(E(\Delta))
 =
\bigoplus_{\vec{c}\in\Delta}
\rH_{i_{\vec{c}}\,}(\link_{\vec{c}\,}(\Delta))\otimes S^{B^{*}(\vec{c}\,)}$
to $\rH_{i}(Q^{*}_{\Delta})$
using Lemma~\ref{lemma_passing_to_homology}.

\section{The main theorem}
\label{section_base_case}

We can now explicitly state the isomorphism of 
Theorem~\ref{theorem_main_result}.
First we introduce notation for the
right-hand side of this theorem.
\begin{definition}
Let
$K_{i}(\Delta)$
denote the direct sum
$\bigoplus_{\vec{c}\, \in \Delta}
\rH_{i_{\vec{c}\,}}(\link_{\vec{c}\,}(\Delta))
                                 \otimes S^{B^{*}(\vec{c}\,)}$.
\label{definition_K}
\end{definition}
A sharpening of
Theorem~\ref{theorem_main_result}
is the following result.
\begin{theorem}
Let $\Delta$ be a subcomplex of $\Comp(n)$.
Then the map 
$$\phi^{\Delta}_{i}:K_{i}(\Delta)\longrightarrow \rH_{i}(Q^{*}_{\Delta})$$
is an $\SSSS_{n-1}$-equivariant isomorphism.
\label{theorem_main}
\end{theorem}

Note that Lemma~\ref{lemma_chain_homology} tells us that the homology of the complex~$E(\Delta)$ is~$K(\Delta)$,
that is,
for all~$i$ we have
$\rH_{i}(E(\Delta))\cong_{\SSSS_{n-1}}K_{i}(\Delta)$.
Lemma~\ref{lemma_commute_diagram_1} implies that  
equation~\eqref{equation_phi} is a well-defined map
from
the homology of~$E(\Delta)$
to the homology groups~$\rH_{i}(Q_{\Delta}^{*})$.

We first prove Theorem~\ref{theorem_main}
in the case when $\Delta$ is a simplex.
This is the case when $\Delta$ is generated by one
composition.

\begin{proposition}
Assume that $\Delta$ is a filter in $\Comp(n)$
generated by one composition,
that is, $\Delta$ is a simplex.
Then Theorem~\ref{theorem_main} holds for~$\Delta$.
\label{proposition_simplex}
\end{proposition}
\begin{proof}
Suppose that $\Delta\subseteq \Comp(n)$ is generated by the
composition $\vec{d}=(d_{1},d_{2}, \ldots, d_{k})$.
Theorem~\ref{theorem_Ehrenborg_Jung}
states that $Q_{\Delta}^{*}$ only has
reduced homology in dimension $k-2$.
Additionally, it states that the action
of $\SSSS_{n-1}$ on the top homology of~$Q_{\Delta}^{*}$
is given by the border
shape Specht module~$S^{B^{*}(\vec{c}\,)}$, that is,
$\rH_{k-2}(Q_\Delta)\cong_{\SSSS_{n-1}} S^{B^{*}(\vec{c}\,)}$.

Next we show that
$\phi_{i}^{\Delta}:K_{i}(\Delta)\longrightarrow\rH_{i}(Q_{\Delta}^{*})$
is an isomorphism for all $i$.
When $i \neq k-2$ both sides are the trivial module,
that is,
$K_{i}(\Delta) = 0 = \rH_{i}(Q_{\Delta}^{*})$
and the map $\phi_{i}^{\Delta}$ is directly
an isomorphism.
Now assume that $i=k-2$. 
Since all the links
$\link_{\vec{c}}(\Delta)$
for $\vec{c} <^{*} \vec{d}$ are contractible,
we have
$$
K_{k-2}(\Delta)
= \bigoplus_{\vec{c}\in\Delta}
  \rH_{k-2-|\vec{c}\,|+1}(\link_{\vec{c}}(\Delta))\otimes S^{B^{*}(\vec{c}\,)}
= \rH_{-1}(\link_{\vec{d}\,}(\Delta)) \otimes S^{B^{*}(\vec{d}\,)} .
$$
Notice that $\link_{\vec{d}\,}(\Delta)$ consists only of
the composition $\vec{d}$ itself,
so that the $(-1)$-dimensional 
reduced homology group
$\rH_{-1}(\link_{\vec{d}\,}(\Delta))$
is the homology of the chain space
$C_{-1}(\link_{\vec{d}\,}(\Delta))$,
which is the one dimensional vector space
with the generator~$\vec{d}$.
Therefore, the map
$\phi_{k-2}^{\Delta}:
   \rH_{-1}(\link_{\vec{d}\,}(\Delta))\otimes S^{B^{*}(\vec{d})}
         \longrightarrow
   \rH_{k-2}(Q^{*}_{\Delta})$
is given by
\begin{align*}
\vec{d} \otimes e_{t}
& =
\vec{d} \otimes \left(\sum_{\gamma \in \SSSS^{C}_{\vec{d}}}
(-1)^{\gamma} \cdot [\gamma\cdot t]\right)
\longmapsto 
\sum_{\gamma\in \SSSS^{C}_{\vec{d}}}
(-1)^{\gamma}
  \cdot
\sigma(\alpha\cdot\gamma,\vec{d}\,) .
\end{align*}
But this is an isomorphism by
Theorem~\ref{theorem_Ehrenborg_Jung}.
\end{proof}

As a direct corollary we have that
Theorem~\ref{theorem_main}
holds for the empty simplex~$\{(n)\}\subseteq\Comp(n)$.
\begin{corollary}
Theorem~\ref{theorem_main} holds for 
the empty simplicial complex,
that is, the simplicial complex consisting only of the composition~$(n)$.
\label{corollary_empty}
\end{corollary}
\begin{proof}
Apply Proposition~\ref{proposition_simplex}
to the simplicial complex $\Delta$
generated by the composition~$(n)$ in $\Comp(n)$.
\end{proof}

\section{The building step}
\label{section_the_building_step}

A simplicial complex which is not a simplex
is the union of smaller simplicial complexes.
We now prove that Theorem~\ref{theorem_main}
holds for the complex $\Delta \cup \Gamma$,
assuming that Theorem~\ref{theorem_main} holds for
the simplicial complexes~$\Delta$,
$\Gamma$, as well as
the intersection~$\Delta \cap \Gamma$.
We build up in the isomorphism~$\phi_{i}^{\Delta\cup\Gamma}$
between
$K_{i}(\Delta\cup\Gamma)$
and
$\rH_{i}(Q^{*}_{\Delta\cup\Gamma})$
from
the associated isomorphisms holding for the smaller
complexes.

\begin{lemma}
\label{lemma_link}
The following two identities hold for the link:
$$
\link_{\vec{c}\,}(\Delta\cap\Gamma)
=
\link_{\vec{c}\,}(\Delta)\cap\link_{\vec{c}\,}(\Gamma)
\:\:\:\: \text{ and } \:\:\:\:
\link_{\vec{c}\,}(\Delta\cup\Gamma)
=
\link_{\vec{c}\,}(\Delta)\cup\link_{\vec{c}\,}(\Gamma) .
$$
\end{lemma}
\begin{lemma}
\label{lemma_Q_split}
The following two identities hold for the ordered set
partition poset:
$$
Q^{*}_{\Delta \cap \Gamma}
=
Q^{*}_{\Delta} \cap Q^{*}_{\Gamma}
\:\:\:\: \text{ and } \:\:\:\:
Q^{*}_{\Delta \cup \Gamma}
=
Q^{*}_{\Delta} \cup Q^{*}_{\Gamma} .
$$
\end{lemma}
The proofs of these two lemmas are straightforward
and are omitted.

Before we begin the proof of Theorem~\ref{theorem_main},
let us remind ourselves of
Definition~\ref{definition_D_chain_complex}. 
For each composition~$\vec{c}$ in~$\Delta$
we have the chain complex~$D^{\vec{c}}(\Delta)$
whose $i$th chain group is
$D^{\vec{c}}_{i}(\Delta)
=
C_{i_{\vec{c}}}(\link_{\vec{c}}(\Delta)) \otimes M^{B^{\#}(\vec{c}\,)}$.
Furthermore, 
$D(\Delta)$ is the chain complex obtained by
taking the direct sum
of~$D^{\vec{c}}(\Delta)$,
where $\vec{c}$ ranges over all
compositions in~$\Delta$.

We now begin the proof of Theorem~\ref{theorem_main}.
\begin{lemma}
\label{lemma_exact_intersection}
For $\vec{c}\in\Delta\cap\Gamma$ the following diagram is commutative, and its rows are exact.
$$
\xymatrix{
0\ar[r]
&D^{\vec{c}}_{i}(\Delta\cap\Gamma)\ar[d]^{\phi_{i}^{\Delta\cap\Gamma,\vec{c}}}\ar[r]
&D^{\vec{c}}_{i}(\Delta)\oplus D^{\vec{c}}_{i}(\Gamma)\ar[d]^{\phi_{i}^{\Delta,\vec{c}}\oplus\phi_{i}^{\Gamma,\vec{c}}}\ar[r]
&D^{\vec{c}}_{i}(\Delta\cup\Gamma)\ar[d]^{\phi_{i}^{\Delta\cup\Gamma,\vec{c}}}\ar[r]
&0\\
0\ar[r]&C_{i}(Q^{*}_{\Delta\cap\Gamma})\ar[r]
&C_{i}(Q^{*}_{\Delta}) \oplus C_{i}(Q^{*}_{\Gamma})\ar[r]
&C_{i}(Q^{*}_{\Delta\cup\Gamma})\ar[r]
&0}
$$
\end{lemma}
\begin{proof}
The horizontal maps in the above diagram are given by
the construction of the Mayer--Vietoris sequence applied to
$\link_{\vec{c}\,}(\Delta\cup\Gamma)=\link_{\vec{c}\,}(\Delta)\cup\link_{\vec{c}\,}(\Gamma)$
in the top row,
and
$Q^{*}_{\Delta\cup\Gamma}=Q^{*}_{\Delta}\cup Q^{*}_{\Gamma}$
in the bottom row.
The top horizontal maps have also been tensored with
the identity map on the Specht modules.
As the Specht module is free, both the top and bottom rows of the diagram remain exact.

We show commutativity of the left square, as the right square
is analogous. 
Let 
$\vec{d}\otimes\alpha\in C_{i_{\vec{c}}}(\Delta\cap\Gamma)\otimes M^{B^{\#}(\vec{c}\,)}$
be a basis element.
First we trace right then down to obtain:
$$
\vec{d}\otimes\alpha
\longmapsto
(\vec{d}\otimes\alpha)\oplus-(\vec{d}\otimes\alpha)
\longmapsto
\sigma(\alpha,\vec{d}\,)\oplus-\sigma(\alpha,\vec{d}\,)  .
$$
We obtain the same result by first tracing
down then right:
\begin{align*}
\vec{d}\otimes\alpha
& \longmapsto
\sigma(\alpha,\vec{d}\,)
\longmapsto
\sigma(\alpha,\vec{d}\,)\oplus-\sigma(\alpha,\vec{d}\,) .
\qedhere
\end{align*}
\end{proof}

\begin{lemma}
\label{lemma_exact_intersection_4}
For each $\vec{c}\in\Delta-\Gamma$,
we have the commutative diagram with exact rows:
$$
\xymatrix{
0\ar[r]
&0\ar[r]\ar[d]^{0}
&D_{i}^{\vec{c}}(\Delta)\ar[r]^{\id}\ar[d]^{\phi_{i}^{\Delta}\oplus 0}
&D_{i}^{\vec{c}}(\Delta)\ar[r]\ar[d]^{\phi_{i}^{\Delta\cup\Gamma}}
&0\\
0\ar[r]
&C_{i}(Q^{*}_{\Delta\cap\Gamma})\ar[r]
&C_{i}(Q^{*}_{\Delta})\oplus C_{i}(Q^{*}_{\Gamma})\ar[r]
&C_{i}(Q^{*}_{\Delta\cup\Gamma})\ar[r]
&0}
$$
\end{lemma}
\begin{proof}
The left-hand square is trivially commutative.
We show the right-hand square commutes
by first tracing right then down:
$$ \vec{c} \otimes \pi
\longmapsto
     \vec{c} \otimes \pi
\longmapsto
     \sigma(\pi,\vec{c}\,) . $$
Now we trace down then right:
\begin{align*}\vec{c}\otimes\pi\longmapsto\sigma(\pi,\vec{c}\,)\oplus 0\longmapsto\sigma(\pi,\vec{c}\,)+0=\sigma(\pi,\vec{c}\,)
\end{align*}
Exactness of the rows in the diagram follows from
Lemma~\ref{lemma_exact_intersection},
as the bottom row has remained unchanged.
\end{proof}

\begin{lemma}
\label{lemma_exact_intersection_3}
The following diagram is commutative, and its rows are exact.
$$
\xymatrix{
0\ar[r]
&D_{i}(\Delta\cap\Gamma)\ar[d]^{\phi_{i}^{\Delta\cap\Gamma}}\ar[r]
&D_{i}(\Delta)\oplus D_{i}(\Gamma)\ar[d]^{\phi_{i}^{\Delta}\oplus\phi_{i}^{\Gamma}}\ar[r]
&D_{i}(\Delta\cup\Gamma)\ar[d]^{\phi_{i}^{\Delta\cup\Gamma}}\ar[r]
&0\\
0\ar[r]
&C_{i}(Q^{*}_{\Delta\cap\Gamma})\ar[r]
&C_{i}(Q^{*}_{\Delta})\oplus C_{i}(Q^{*}_{\Gamma})\ar[r]
&C_{i}(Q^{*}_{\Delta\cup\Gamma})\ar[r]
&0}
$$
\end{lemma}
\begin{proof}
The proof is to take
direct sums of the previous two short exact sequences.
First, take the direct sum of the diagram in
Lemma~\ref{lemma_exact_intersection}
for each $\vec{c}\in\Delta\cap\Gamma$.
Next, take the resulting short exact sequence of chain complexes
and take its direct sum with the diagram in
Lemma~\ref{lemma_exact_intersection_4}
for each $\vec{c}\in\Delta-\Gamma$.
Finally, switch $\Delta$ and $\Gamma$ in
Lemma~\ref{lemma_exact_intersection_4}
and take the direct sum of the resulting diagram with the diagram from
Lemma~\ref{lemma_exact_intersection_4}
for each $\vec{c}\in\Gamma-\Delta$. Observe that the second row of
the diagram remains the same throughout this process.
Also, note that the top row is exact
as it is the direct sum of exact sequences.
All together, this yields the desired commutative diagram. 
\end{proof}

\begin{proposition}
\label{proposition_D_to_E}
The following diagram is commutative, and its rows are exact.
$$
\xymatrix{
0\ar[r]
&E_{i}(\Delta\cap\Gamma)\ar[d]^{\phi_{i}^{\Delta\cap\Gamma}}\ar[r]
&E_{i}(\Delta)\oplus E_{i}(\Gamma)\ar[d]^{\phi_{i}^{\Delta}\oplus\phi_{i}^{\Gamma}}\ar[r]
&E_{i}(\Delta\cup\Gamma)\ar[d]^{\phi_{i}^{\Delta\cup\Gamma}}\ar[r]
&0\\
0\ar[r]
&C_{i}(Q^{*}_{\Delta\cap\Gamma})\ar[r]
&C_{i}(Q^{*}_{\Delta})\oplus C_{i}(Q^{*}_{\Gamma})\ar[r]
&C_{i}(Q^{*}_{\Delta\cup\Gamma})\ar[r]
&0}
$$
\end{proposition}
\begin{proof}
Since $E_{i}(\Delta)$ is a subspace of $D_{i}(\Delta)$,
it follows from Lemma~\ref{lemma_exact_intersection_3}
that the diagram is commutative.
Furthermore, that the
second row is exact also follows from this lemma.
It remains to show that the first row is exact.
However, this follows by the same reasoning that the first row
of Lemma~\ref{lemma_exact_intersection_3} is exact,
but with the permutation module~$M^{B^{\#}(\vec{c}\,)}$
replaced with the Specht module~$S^{B^{*}(\vec{c}\,)}$.
\end{proof}

\begin{proposition}
Assume that Theorem~\ref{theorem_main}
holds for the simplicial complexes
$\Delta$, $\Gamma$, and 
the intersection $\Delta \cap \Gamma$.
Then Theorem~\ref{theorem_main}
also holds for the union~$\Delta \cup \Gamma$.
\label{proposition_the_building_step}
\end{proposition}
\begin{proof}
Consider the diagram of short exact sequences of chain complexes
given in Proposition~\ref{proposition_D_to_E}.
Use the zig-zag lemma to obtain 
the Mayer--Vietoris sequence:
$$
\xymatrix{
\cdots\ar[r]
&K_{i}(\Delta\cap\Gamma)\ar[d]^{\phi_{i}^{\Delta\cap\Gamma}}\ar[r]
&K_{i}(\Delta)\oplus K_{i}(\Gamma)\ar[d]^{\phi_{i}^{\Delta}\oplus\phi_{i}^{\Gamma}}\ar[r]
&K_{i}(\Delta\cup\Gamma)\ar[d]^{\phi_{i}^{\Delta\cup\Gamma}}\ar[r]
&\cdots\\
\cdots\ar[r]
&\rH_{i}(Q^{*}_{\Delta\cap\Gamma})\ar[r]
&\rH_{i}(Q^{*}_{\Delta}) \oplus \rH_{i}(Q^{*}_{\Gamma})\ar[r]
&\rH_{i}(Q^{*}_{\Delta\cup\Gamma})\ar[r]
&\cdots\\}
$$
The assumption that
Theorem~\ref{theorem_main}
holds for
the complexes
$\Delta \cap \Gamma$, $\Delta$, and $\Gamma$
implies that
$\phi_{i}^{\Delta\cap\Gamma}$
and
$\phi_{i}^{\Delta} \oplus \phi_{i}^{\Gamma}$
are isomorphisms.
The five-lemma now implies that
$\phi_{i}^{\Delta\cup\Gamma}$ is also an isomorphism.
Furthermore, $\phi_{i}^{\Delta\cup\Gamma}$ is an $\SSSS_{n-1}$-equivariant map
by Lemma~\ref{lemma_equivariance}. 
\end{proof}

\begin{proof}[Proof of Theorem~\ref{theorem_main}.]
Since every simplicial complex $\Delta$ is 
the union of simplexes, 
Proposition~\ref{proposition_the_building_step}
implies that it is enough to prove
Theorem~\ref{theorem_main}
for simplexes and the empty simplex.
This was done in
Proposition~\ref{proposition_simplex}
and
Corollary~\ref{corollary_empty}.
\end{proof}

\section{Alternate Proof of Theorem~\ref{theorem_main_result}}

As mentioned in the introduction,
we now give an alternate proof of Theorem~\ref{theorem_main_result} 
using a poset fiber theorem of
Bj\"orner, Wachs and Welker~\cite{Bjorner_Wachs_Welker}.

\begin{theorem}
Let $\Delta$ be a simplicial complex of compositions of $n$.
Then the $i$th reduced homology group
of the simplicial complex~$Q^{*}_{\Delta}$
is given by
$$
\rH_{i}(Q^{*}_{\Delta})
\cong_{\SSSS_{n-1}}
\bigoplus_{\vec{c}\, \in \Delta}
       \rH_{i-|\vec{c}\,|+1}(\link_{\vec{c}\,}(\Delta))
   \otimes
          S^{B^{*}(\vec{c}\,)}.    $$
\end{theorem}
\begin{proof}
Consider 
the two posets
$\Delta$ and~$Q^{*}_{\Delta}$
with the reverse order $\leq^{*}$
and the poset map
$$ \type : Q^{*}_{\Delta} - \{([n])\} \longrightarrow \Delta - \{(n)\} . $$
Observe that the type map respects the action of
the symmetric group~$\SSSS_{n-1}$.
Now the inverse image
$\type^{-1}(\Delta_{\leq^{*} \vec{c}})$
is the filter~$Q^{*}_{\vec{c}}$.
Since $Q^{*}_{\vec{c}}$ only has reduced homology
in dimension $|\vec{c}\,| - 2$ by
Theorem~\ref{theorem_Ehrenborg_Jung}, 
we have that
the fiber
$\triangle(\type^{-1}(\Delta_{\leq^{*} \vec{c}}))$
is
$(|\vec{c}\,| - 3)$-acyclic,
where $|\vec{c}\,| - 3$ is the length of the longest chain
in $\type^{-1}(\Delta_{<^{*} \vec{c}})$.
Hence 
Theorem~9.1 of~\cite{Bjorner_Wachs_Welker}
applies.
Since~$\SSSS_{n-1}$ acts trivially on $\Delta$
(see Remark~\ref{remark_action}),
we have that the stabilizer
$\Stab_{\SSSS_{n-1}}(\vec{c}\,)$
is in fact the whole group $\SSSS_{n-1}$.
Thus there is no representation to induce and we have
\begin{align*}
\rH_{i}(Q^{*}_{\Delta})
& \cong_{\SSSS_{n-1}}
\rH_{i}(\Delta)
\oplus
\bigoplus_{\vec{c}\, \in \Delta - \{(n)\}}
       \rH_{|\vec{c}\,|-2}(\type^{-1}(\Delta_{\leq^{*} \vec{c}}))
   \otimes
        \rH_{i - |\vec{c}\,| + 1}((\Delta - \{(n)\})_{>^{*} \vec{c}})  \\
& \cong_{\SSSS_{n-1}}
\rH_{i}(\Delta)
\oplus
\bigoplus_{\vec{c}\, \in \Delta - \{(n)\}}
          S^{B^{*}(\vec{c}\,)}
   \otimes          
       \rH_{i-|\vec{c}\,|+1}(\link_{\vec{c}\,}(\Delta))   ,
\end{align*}
where the first summand corresponds to $\vec{c} = (n)$
and the trivial representation $S^{B^{*}(n)}$, proving the result.
\end{proof}

\section{Filters in the set partition lattice}
\label{section_filters_in_the_set_partition_lattice}

In Theorem~\ref{theorem_main_result} we characterized
each homology group of~$Q^{*}_{\Delta}$,
a subspace of ordered set partitions.
We will now translate the topological data we have gathered on $Q_{\Delta}^{*}$ into data on the usual partition lattice~$\Pi_{n}$.

Recall that $Q_{\Delta}^{*}$ is the collection of ordered set partitions
containing the element $n$ in the last block, whose type is contained
in the simplicial complex $\Delta\subseteq \Comp(n)$.
Define the \emph{forgetful map}
$f:Q^{*}_{\Delta} \longrightarrow \Pi_{n}$
given by removing the order between blocks, that is,
$f((C_{1},C_{2},\ldots,C_{k})) = \{C_{1},C_{2},\ldots,C_{k}\}$.
\begin{definition}
Let $\Pistar_{\Delta}\subseteq \Pi_{n}$
be the image of $Q_{\Delta}^{*}$ under the forgetful map~$f$.
\end{definition}

\begin{lemma}
\label{lemma_integer_filter}
Suppose that $F$ is a filter in the integer partition lattice.
Let $\Delta_{F}$ be the
filter of compositions given by
$\{\vec{c} \in \Comp(n) : \type(\vec{c}\,)\in F\}$.
Then the associated filter~$\Pistar_{\Delta_{F}}$
in the partition lattice is given by
$\{\pi \in \Pi_{n} : \type(\pi)\in F\}$.
\end{lemma}
\begin{proof}
Choose $\pi \in \Pi_{n}$ such that $\type(\pi)\in F$,
with $\pi=\{B_{1},B_{2}, \ldots, B_{k}\}$
where we assume $n \in B_{k}$.
The ordered set partition $\tau=(B_{1},B_{2}, \ldots, B_{k})$
is an element of $Q_{\Delta_{n}}^*$,
since $\type(\tau) = \type(\pi) \in F$.
Hence $\pi$ is in the image of the forgetful map~$f$.
The other direction is clear.
\end{proof}
\begin{remark}
{\rm
In general, taking the image of a filter $\Delta\subseteq\Comp(n)$ 
under the map $\type$ does not define a filter
in the integer partition lattice~$I_{n}$. 
For example,
consider the simplex~$\Delta$ in~$\Comp(6)$ generated by~$(3,2,1)$.
Note that
$\type(\Delta)$ consists of the four partitions
$\{\{3,2,1\}, \{3+2,1\}, \{3,2+1\}, \{3+2+1\}\}
=
\{\{3,2,1\}, \{5,1\}, \{3,3\}, \{6\}\}$.
This is not a filter in $I_{6}$ since it does not contain
the partition~$\{4,2\}$.
}
\end{remark}
\begin{lemma}
\label{lemma_forgetful_equivariant}
The forgetful map $f:Q_{\Delta}^{*}\longrightarrow\Pistar_{\Delta}$
respects the $\SSSS_{n-1}$-action.
\end{lemma}
\begin{proof}
Let $\alpha \in \SSSS_{n-1}$
and $\sigma=(C_{1}, \ldots, C_{k})\in Q_{\Delta}^{*}$. 
Then we have that 
\begin{align*}
f(\alpha \cdot \sigma)
 & =
f((\alpha(C_{1}),\ldots,\alpha(C_{k})))
=
\{\alpha(C_{1}),\ldots,\alpha(C_{k})\}
=
\alpha \cdot f(\sigma).
\qedhere
\end{align*}
\end{proof}

The $\SSSS_{n-1}$ action on $\Pistar_{\Delta}$
extends to the chains in the order complex
$\triangle(\Pistar_{\Delta}-\{\ho\})$.

For a statement of the equivariant version of the Quillen Fiber Lemma,
see~\cite[Theorem~5.2.2]{Wachs_III}.

\begin{proposition}
\label{proposition_forgetful}
The forgetful map
$f : Q^{*}_{\Delta} - \{\ho\} 
   \longrightarrow \Pistar_{\Delta} - \{\ho\} = P$
satisfies the condition of
Quillen's Equivariant Fiber Lemma, that is,  
for a partition $\pi = \{B_{1}, B_{2}, \ldots, B_{k}\}$ 
in~$P$,
the order complex
$\triangle(f^{-1}(P_{\geq \pi}))$ is the barycentric subdivision of a cone, and is
therefore contractible and acyclic.
\end{proposition}
\begin{proof}
Let $B_{k}$ be the block of the partition~$\pi$
that contains the element~$n$.
Note that because every ordered partition
in~$Q_{\Delta}^{*}$ must have the element $n$ in its last block,
we must have that each ordered set partition in the fiber
$f^{-1}(\pi)$ has the set~$B_{k}$ as its last block.
Furthermore, 
the last block of each ordered set partition in $f^{-1}(P_{\geq \pi})$
contains the block~$B_{k}$.

We claim that $f^{-1}(P_{\geq \pi})$ is a cone
with apex $([n] - B_{k}, B_{k})$.
Let $\sigma \in f^{-1}(P_{\geq \pi})$ be the ordered set partition
$\sigma=(C_{1}, \ldots, C_{p-1}, C_{p})$.
Note that the number of blocks of $\sigma$,
is greater than or equal to~$2$ as we have removed
the maximal element~$\ho$ from~$Q_{\Delta}^{*}$.
If $C_{p} = B_{k}$ then the face $\sigma$
contains the vertex $([n] - B_{k}, B_{k})$.
If $C_{p} \supsetneq B_{k}$ then both
$\sigma$ and the vertex $([n] - B_{k}, B_{k})$
are contained in 
the face
$(C_{1}, \ldots, C_{p-1}, C_{p} - B_{k}, B_{k})$
in~$f^{-1}(P_{\geq \pi})$.
Thus
$f^{-1}(P_{\geq \pi})$
is the face poset of a cone with vertex $([n]-B_{k},B_{k})$, and therefore $\triangle(f^{-1}(P_{\geq\pi})$
is the barycentric subdivision of a cone and hence contractible and acyclic.
\end{proof}

Combining Proposition~\ref{proposition_forgetful}
with Theorem~\ref{theorem_main_result},
we have the following result for the homology of the order complex $\triangle(\Pistar_{\Delta}-\{\ho\})$.
\begin{theorem}
\label{theorem_main_theorem_order_complex}
The $i$th reduced homology group of the order complex of $\Pistar_{\Delta}-\{\ho\}$ as an $\SSSS_{n-1}$-module
is given by 
$$   \rH_{i}(\triangle(\Pistar_{\Delta} - \{\ho\}))
\cong_{\SSSS_{n-1}}
       \bigoplus_{\vec{c} \in \Delta}
           \rH_{i-|\vec{c}\,|+1}(\link_{\vec{c}\,}(\Delta))
                \otimes
                                 S^{B^{*}(\vec{c}\,)}. $$
\end{theorem}

\begin{remark}
{\rm
Suppose that $\link_{\vec{c}\,}(\Delta)$
has reduced homology in dimension $j$.
By Theorem~\ref{theorem_main_theorem_order_complex}
this reduced homology contributes to
dimension $j+|\vec{c}\,|-1$
of the reduced homology of the order complex
of~$\Pistar_{\Delta}-\{\ho\}$.
}
\label{remark_dimension_shift}
\end{remark}

We end the section with a discussion of Morse matchings
in the link~$\link_{\vec{c}\,}(\Delta)$.
Assume that the link $\link_{\vec{c}}(\Delta)$
has a discrete Morse matching
with critical cell~$\vec{d}$, which also contributes
to the reduced homology of $\link_{\vec{c}}(\Delta)$.
For instance, this case occurs if $\vec{d}$ is
a facet.
Similarly, $\vec{d}$ will contribute to the reduced homology
of $\link_{\vec{c}}(\Delta)$ if $\vec{d}$ is a homology
facet of a shelling.
In either case, the critical cell~$\vec{d}$ contributes to
the reduced homology of
$\triangle(\Pistar_{\Delta} - \{\ho\})$
in dimension
$\dim_{\link_{\vec{c}}(\Delta)}(\vec{d}\,) + |\vec{c}\,| - 1
=
\dim_{\Delta}(\vec{d}\,)
= |\vec{d}\,| - 2$,
by equation~\eqref{equation_dimension}.
Note that this dimension is independent of the
composition~$\vec{c}$.

\section{Consequences of the main result}
\label{section_consequences}

As the title of this section suggests, we will now derive results from Theorem~\ref{theorem_main_theorem_order_complex} using topological data from~$\Delta$.
\begin{theorem}
\label{theorem_manifold_homology}
Assume that $\Delta$ is homeomorphic to a $k$-dimensional manifold with or without boundary.
Then the reduced homology of
the order complex $\triangle(\Pistar_{\Delta} - \{\ho\})$
is given by
$$
 \rH_{i}(\triangle(\Pistar_{\Delta} - \{\ho\}))
       \cong_{\SSSS_{n-1}}
                 \rH_{i}(\Delta)
              \tensor
                1_{\SSSS_{n-1}}      \:\:\:\:\: \text{ for } i < k, $$
and the top dimensional homology is given by
$$
\rH_{k}(\triangle(\Pistar_{\Delta} - \{\ho\}))
       \cong_{\SSSS_{n-1}}
                 \rH_{k}(\Delta)
\tensor
               1_{\SSSS_{n-1}}
      \oplus
      \bigoplus_{\vec{c}\, \in\, \Int({\Delta})}
                S^{B^{*}(\vec{c}\,)}, 
$$
where $1_{\SSSS_{n-1}}$ is the trivial representation of $\SSSS_{n-1}$, and the direct sum is over the interior faces of the manifold~$\Delta$.
Moreover, these isomorphisms hold as
$\SSSS_{n-1}$-modules.
\end{theorem}
\begin{proof}
Since $\Delta$ is homeomorphic to a $k$-dimensional manifold,
we may apply the comment preceding Proposition~3.8.9 of~\cite{EC1},
which states that for any $\vec{c}\in\Delta$,
where $\vec{c}$ is not the empty composition~$(n)$,
we have that $\link_{\vec{c}\,}(\Delta)$ has the homology groups
of a sphere of dimension $k-|\vec{c}\,|+1$
if $\vec{c}$ is on the interior of~$\Delta$,
or the homology groups of a ball of dimension $k-|\vec{c}\,|+1$
if $\vec{c}$ is on the boundary.
Hence if $\vec{c}$ is on the boundary of $\Delta$ it does not contribute
to the reduced homology of $\Delta(\Pistar_{\Delta}-\{\ho\})$. 
If instead $\vec{c}$ is in the interior of~$\Delta$
then
by Remark~\ref{remark_dimension_shift}
it will contribute to the 
reduced homology group of
dimension $(k-|\vec{c}\,|+1)+|\vec{c}\,|-1=k$.
This is the top homology of the complex.
Finally, observe that the composition~$(n)$
contributes to all homology groups
of $\triangle(\Pistar_{\Delta}-\{\ho\})$
when $\Delta$ has nontrivial homology,
and that the Specht module
$S^{B^{*}(n)}$ is the trivial representation~$1_{\SSSS_{n-1}}$.
\end{proof}

We now give two immediate corollaries of
Theorem~\ref{theorem_manifold_homology},
when $\Delta$ is
homeomorphic to a sphere or a ball.
\begin{corollary}
\label{corollary_sphere}
Suppose that $\Delta$ is homeomorphic to a sphere of dimension $k$. Then the order complex $\triangle(\Pistar_{\Delta}-\{\ho\})$ only has homology in dimension $k$ given by
$$\rH_{k}(\triangle(\Pistar_{\Delta}-\{\ho\}))
\cong_{\SSSS_{n-1}}
\bigoplus_{\vec{c}\in\Delta}S^{B^{*}(\vec{c}\,)}.$$
\end{corollary}
\begin{corollary}
\label{corollary_ball}
Suppose that $\Delta$ is homeomorphic to a ball of dimension $k$. Then the order complex $\triangle(\Pistar_{\Delta}-\{\ho\})$ only has homology in dimension $k$ given by
$$\rH_{k}(\triangle(\Pistar_{\Delta}-\{\ho\}))
\cong_{\SSSS_{n-1}}
\bigoplus_{\vec{c}\,\in\Int(\Delta)}S^{B^{*}(\vec{c}\,)}.$$
\end{corollary}
Next we obtain a result about
$\triangle(\Pistar_{\Delta}-\{\hat{1}\})$ 
when $\Delta$ is shellable.
\begin{proposition}
\label{proposition_pure_shellable}
Suppose that $\Delta$ is a shellable complex of dimension $k$.
Then the order complex $\triangle(\Pistar_{\Delta}-\{\ho\})$
only has reduced homology in dimension $k$
given by
$$\rH_{k}(\triangle(\Pistar_{\Delta}-\{\ho\}))
   \cong_{\SSSS_{n-1}}
\bigoplus_{\vec{c}\,\in \Delta}
\rbeta_{k - |\vec{c}\,| + 1}(\link_{\vec{c}}(\Delta)) \cdot
S^{B^{*}(\vec{c}\,)}  . $$
\end{proposition}
\begin{proof}
Note that the face $\vec{c}$ has dimension
$|\vec{c}\,|-2$. Hence the link
$\link_{\vec{c}}(\Delta)$ has dimension
$k-\dim(\vec{c}\,)-1 = k - |\vec{c}\,| + 1$,
by equation~\eqref{equation_dimension}.
Since the link is shellable, all of its reduced homology
occurs
in dimension $k - |\vec{c}\,| + 1$ and this contributes
only to the reduced homology of dimension~$k$
of $\triangle(\Pistar_{\Delta} - \{\ho\})$ by
Remark~\ref{remark_dimension_shift}.
Lastly, the betti number $\rbeta_{k-|\vec{c}\,|+1}$ 
is explained by the fact that a shellable complex has the homotopy type
of a wedge of spheres of the same dimension.
\end{proof}

\section{The representation ring}
\label{section_representation_ring}

The {\em representation ring} $R(G)$ of a group $G$
is the free abelian group with
generators given by representations $V$ of~$G$
modulo the subgroup generated by $V + W - V \oplus W$.
Elements of the representation ring are called virtual representations 
because summands can have negative coefficients. For finite groups, 
complete reducibility implies $R(G)$ is just the free abelian group
generated by 
the irreducible representations~$V$ of~$G$.

\begin{remark}
\label{remark_trivial_action}
{\rm
Suppose $G$ acts trivially on the space $V$.
Then $V \otimes W \cong_{G} \dim(V) \cdot W$
in the representation ring~$R(G)$.
}
\end{remark}

\begin{proof}
Since $G$ acts trivially on $V$ we know that
$V \cong_{G} \mathbb{C}^{\dim(V)}$.
Thus,
$V\otimes W \cong_{G} \mathbb{C}^{\dim(V)} \otimes W
\cong_{G} \dim(V) \cdot W.$
\end{proof}

In the representation ring we can compute the
alternating sum of the homology groups
of $\triangle(\Pistar_{\Delta}-\{\ho\})$,
which we do in the following proposition.
This can be seen as $\SSSS_{n-1}$-analogue of
the reduced Euler characteristic.
\begin{proposition}
\label{proposition_virtual}
As virtual $\SSSS_{n-1}$-representations we have that
$$ \bigoplus_{i \geq -1} (-1)^{i} \cdot \rH_{i}(\triangle(\Pistar_{\Delta}-\{\ho\}))
\cong
\bigoplus_{\vec{c}\, \in \Delta}
(-1)^{|\vec{c}\,| - 1} \cdot
\rchi(\link_{\vec{c}\,}(\Delta))
   \cdot
        S^{B^{*}(\vec{c}\,)}   .
$$
\end{proposition}
\begin{proof}
We begin the proof by applying alternating sums to both sides of
Theorem~\ref{theorem_main_theorem_order_complex}.
\begin{align*}
\bigoplus_{i \geq -1} (-1)^{i} \cdot \rH_{i}(\triangle(\Pistar_{\Delta}-\{\hat{1}\}))
& \cong
\bigoplus_{i \geq -1} (-1)^{i} \cdot
\bigoplus_{\vec{c}\, \in \Delta}
       \rH_{i-|\vec{c}\,|+1}(\link_{\vec{c}\,}(\Delta))
   \otimes
        S^{B^{*}(\vec{c}\,)}   \\
& \cong
\bigoplus_{\vec{c}\, \in \Delta}
\bigoplus_{i \geq -1} (-1)^{i} \cdot
       \rbeta_{i-|\vec{c}\,|+1}(\link_{\vec{c}\,}(\Delta))
   \cdot
        S^{B^{*}(\vec{c}\,)}   \\
& \cong
\bigoplus_{\vec{c}\, \in \Delta}
(-1)^{|\vec{c}\,| - 1} \cdot
\bigoplus_{j \geq -1} (-1)^{j} \cdot
       \rbeta_{j}(\link_{\vec{c}\,}(\Delta))
   \cdot
        S^{B^{*}(\vec{c}\,)}   \\
& \cong
\bigoplus_{\vec{c}\, \in \Delta}
(-1)^{|\vec{c}\,| - 1} \cdot
\rchi(\link_{\vec{c}\,}(\Delta))
   \cdot
        S^{B^{*}(\vec{c}\,)}   ,
\end{align*}
where the second step is by
Remark~\ref{remark_trivial_action},
since $\SSSS_{n-1}$ acts trivially on
$\rH_{i-|\vec{c}\,|+1}(\link_{\vec{c}\,}(\Delta))$.
In the last step we used that the alternating sum of
the Betti numbers is the reduced Euler characteristic.
\end{proof}

The next lemma is straightforward to prove using jeu-de-taquin;
see~\cite{Sagan} or~\cite[A.1.2]{EC2}. 
\begin{lemma}
\label{lemma_permutation_module_isomorphism}
The permutation module~$M^{B^{\#}(\vec{c}\,)}$
is isomorphic to the direct sum over
all border strip Specht modules~$S^{B^{*}(\vec{d}\,)}$
for~$\vec{d}\leq^{*}\vec{c}$,
that is,
$$ M^{B^{\#}(\vec{c}\,)}
\cong_{\SSSS_{n-1}}
\bigoplus_{\vec{d}\,\leq^{*}\,\vec{c}}S^{B^{*}(\vec{d}\,)} . $$
\end{lemma}
\begin{proof}
Recall that the border strip of shape $A^{\#}(\vec{c}\,)$ was defined
in Section~\ref{section_border_strips}.

We have the isomorphism
$S^{A^{\#}(\vec{c}\,)} \cong_{\SSSS_{n-1}} M^{A^{\#}(\vec{c}\,)}$
because the rows of the shape~$A(\vec{c}/1)$ are non-overlapping,
thus polytabloids of shape~$A(\vec{c}/1)$
are tabloids of shape~$A(\vec{c}/1)$.
Additionally, we have
$M^{A^{\#}(\vec{c}\,)}\cong_{\SSSS_{n-1}}M^{B^{\#}(\vec{c}\,)}$,
since tabloids are defined as row equivalence classes of tableaux
and $A(\vec{c}/1)$ and $B(\vec{c}/1)$ have the same rows.
Combining these two $\mathfrak{S}_{n-1}$-isomorphisms
yields $M^{B^{\#}(\vec{c}\,)}\cong S^{A^{\#}(\vec{c}\,)}$.

Now consider the $k-1$ empty boxes
situated to the left of every row in
the Specht module defined by the shape~$A^{\#}(\vec{c}\,)$,
but above the last box of the previous row.
For each of these boxes
perform
a jeu-de-taquin slide into this box.

For each slide, there are two alternatives.
If the slide is horizontal, it moves
the upper row one step to the left such that
the two rows overlap in one position.
If the slide is vertical then every
entry in the lower row moves one step up.

After performing all the $k-1$ slides the result
is a border shape of shape~$B^{\#}(\vec{c}\,)$,
where the composition~$\vec{c}$ is less
than or equal to the composition $\vec{d}$
in the dual order. 
\end{proof}

Proposition~\ref{proposition_virtual}
can also be proved using the Hopf trace formula;
see~\cite[Theorem 2.3.9]{Wachs_III}.

\begin{proof}[Second proof of Proposition~\ref{proposition_virtual}.]
Recall that
$\rH_{i}(\triangle(\Pistar_{\Delta}-\{\ho\})) \cong \rH_{i}(Q_{\Delta}^{*})$.
By applying the Hopf trace formula we have that
\begin{align*}
\bigoplus_{i \geq -1}
(-1)^{i} \cdot \rH_{i}(Q^{*}_{\Delta})
& \cong
\bigoplus_{i \geq -1}
(-1)^{i} \cdot C_{i}(Q^{*}_{\Delta}) \\
& \cong
\bigoplus_{\vec{d} \in \Delta}
(-1)^{|\vec{d}\,|} \cdot M^{B^{\#}(\vec{d}\,)} \\
& \cong
\bigoplus_{\vec{d} \in \Delta}
(-1)^{|\vec{d}\,|} \cdot
\bigoplus_{\vec{c}\, \leq^{*} \vec{d}}
S^{B^{*}(\vec{c}\,)} \\
& \cong
\bigoplus_{\vec{c} \in \Delta}
\sum_{\onethingatopanother{\vec{d} \geq^{*} \vec{c}}{\vec{d} \in \Delta}}
(-1)^{|\vec{d}\,|} \cdot
S^{B^{*}(\vec{c}\,)} \\
& \cong
\bigoplus_{\vec{c} \in \Delta}
(-1)^{|\vec{c}\,| - 1} \cdot
\sum_{\onethingatopanother{\vec{d} \geq^{*} \vec{c}}{\vec{d} \in \Delta}}
(-1)^{|\vec{d}\,| - |\vec{c}\,| - 1} \cdot
S^{B^{*}(\vec{c}\,)} .
\end{align*}
Notice that in the second isomorphism
we have used that the chain space~$C_{i}(Q_{\Delta}^{*})$ has basis given by 
all ordered set partitions into $i+2$ parts
with type in~$\Delta$.
This is equivalent
to the direct sum over all permutation modules~$M^{B^{\#}(\vec{d}\,)}$ where $\vec{d}\in\Delta$ 
is a composition of~$n$ into $i+2$~parts. The remaining step is to observe that the inner sum of the last line
is given by
the reduced Euler characteristic~$\rchi(\link_{\vec{c}}(\Delta))$. 
\end{proof}

We observe that in the case when
the order complex $\triangle(\Pistar_{\Delta} - \{\ho\})$
has all its reduced homology concentrated in
one dimension,
the second proof of Proposition~\ref{proposition_virtual}
which uses the Hopf trace formula
gives a shorter proof of our main result
Theorem~\ref{theorem_main_theorem_order_complex}.

Lastly, by taking dimension on both sides of
Proposition~\ref{proposition_virtual}
we obtain the reduced Euler characteristic
of~$\triangle(\Pistar_{\Delta} -\{\ho\})$.
\begin{corollary}
\label{corollary_Euler}
The reduced Euler characteristic of $\triangle(\Pistar_{\Delta}-\{\ho\})$
is given by
$$
\rchi(\triangle(\Pistar_{\Delta}-\{\ho\}))
=
\sum_{\vec{c}\, \in \Delta}
(-1)^{|\vec{c}\,| - 1} \cdot
       \rchi(\link_{\vec{c}\,}(\Delta))
   \cdot 
        \beta^{*}_{n}(\vec{c}\,)    .  
$$
\end{corollary}
This corollary extends Theorem~3.1
from~\cite{Ehrenborg_Readdy_I}.

\section{The homotopy type of $\Pistar_{\Delta}$}

We turn our attention to the homotopy type
of the order complex $\triangle(\Pistar_{\Delta} - \{\ho\})$.
By combining the poset fiber theorems
of Quillen~\cite{Quillen}
and Bj\"orner, Wachs and
Welker~\cite{Bjorner_Wachs_Welker}
we obtain the next result.
Recall that $*$ denotes the (free) join of complexes.
\begin{theorem}
The order complex of $\Pistar_{\Delta}-\{\ho\}$
is homotopy equivalent to 
the complex of ordered set partitions~$Q_{\Delta}^{*}$,
that is, 
$\triangle(\Pistar_{\Delta}-\{\ho\}) \simeq Q_{\Delta}^{*}$.
Furthermore, the following
homotopy equivalence holds:
$$
        Q^{*}_{\Delta}
\: \simeq \:
         \triangle(\Delta - \{(n)\})
\: \vee \:
         \{ Q^{*}_{\vec{c}} * \link_{\vec{c}}(\Delta)
               \: : \:
               \vec{c} \in \Delta - \{(n)\}   \}   ,
$$
where $\vee$ denotes identifying each vertex
$\vec{c}$ in $\triangle(\Delta - \{(n)\}$
with any vertex in~$Q^{*}_{\vec{c}}$.
In the case when the complex~$\Delta$ is connected
then
the homotopy equivalence simplifies to
$$
        Q^{*}_{\Delta}
\simeq
         \bigvee_{\vec{c} \in \Delta}
           Q^{*}_{\vec{c}} * \link_{\vec{c}}(\Delta) .
$$ 
\end{theorem}
\begin{proof}
The first homotopy equivalence
follows by applying Quillen's fiber lemma
to the forgetful map $f$.
This yields 
$\triangle(\Pistar_{\Delta}-\{\ho\}) \simeq \triangle(Q_{\Delta}^{*}-\{\ho\})
\cong Q_{\Delta}^{*}$, since 
$\triangle(Q_{\Delta}^{*} - \{\ho\})$
is the barycentric subdivision of~$Q_{\Delta}^{*}$.

The second homotopy equivalence in both cases follows
by Theorem~1.1
in~\cite{Bjorner_Wachs_Welker}, with the same
reasoning as in the proof of Theorem~\ref{theorem_main_result}.
Furthermore,
when $\vec{c} = (n)$ then
the complex~$Q^{*}_{\vec{c}}$ is the empty complex,
which is the identity for the join.
\end{proof}

\begin{corollary}
Let $\Delta$ be a connected
simplicial complex.
Assume furthermore that each link
(including $\Delta$)
$\link_{\vec{c}}(\Delta)$ is a wedge of spheres.
Then the order complex $\triangle(\Pistar_{\Delta}-\{\ho\})$
is also a wedge of spheres.
Furthermore, the number
of $i$-dimensional spheres is given by the
sum
\begin{equation}
     \sum_{\vec{c} \in \Delta}
             \beta^{*}_{n}(\vec{c}\,)
           \cdot
              \rbeta_{i - |\vec{c}\,| + 1}(\link_{\vec{c}\,}(\Delta)) ,
\label{equation_i_spheres}
\end{equation}
where $\rbeta_{j}$ denotes the $j$th reduced Betti number.
\label{corollary_every_link_wedge_spheres}
\end{corollary}

Next we have the homotopy versions
of Corollaries~\ref{corollary_sphere}
and~\ref{corollary_ball}.
To prove the next two corollaries,
we are again using
the comment preceding Proposition~3.8.9 of~\cite{EC1}
to determine the reduced Betti numbers of the links.
\begin{corollary}
\label{corollary_homotopy_sphere}
Suppose that $\Delta$ is homeomorphic to a sphere of dimension $k$. 
Then the order complex $\triangle(\Pistar_{\Delta}-\{\ho\})$ 
is a wedge of $k$-dimensional spheres
and the number of spheres is given by the sum:
$$ \sum_{\vec{c}\in\Delta} \beta^{*}_{n}(\vec{c}\,). $$
\end{corollary}
\begin{corollary}
\label{corollary_homotopy_ball}
Suppose that $\Delta$ is homeomorphic to a ball of dimension $k$.
Then the order complex $\triangle(\Pistar_{\Delta}-\{\ho\})$ 
is a wedge of $k$-dimensional spheres
and the number of spheres is given by the sum:
$$ \sum_{\vec{c}\,\in\Int(\Delta)}  \beta^{*}_{n}(\vec{c}\,). $$
\end{corollary}

We end this section with a discussion of how
we can lift discrete Morse matchings
from the links of~$\Delta$ to the complex
of order set partitions~$Q^{*}_{\Delta}$.
\begin{definition}
For an ordered set partition
$\sigma = (C_{1}, C_{2}, \ldots, C_{k})$ of $n$,
where
$C_{i} = \{c_{i,1} < c_{i,2} < \cdots< c_{i,j_{i}}\}$
and
$|C_{i}| = j_{i}$,
define the permutation
$\perm(\sigma)\in\SSSS_{n}$ to be the elements
of the blocks listed in the order of the blocks,
that is,
$$
\perm(\sigma)
  =
c_{1,1}, c_{1,2}, \ldots, c_{1,j_{1}},
c_{2,1}, c_{2,2}, \ldots,
                                          c_{k,j_{k}} . $$
\end{definition}
Define the descent set of an ordered
set partition $\sigma$ to be
$\Des(\sigma) = \Des(\perm(\sigma))$.
Observe that the descent composition of an
ordered set partition is an order preserving
map from the poset of ordered set partitions~$Q_{n}$
to the poset of compositions~$\Comp_{n}$, that is,
$\Des: Q_{n}^{*} \longrightarrow \Comp(n)$ is a poset map.

\begin{lemma}
\label{lemma_patchwork}
Let $\Delta$ be a filter in the composition poset $\Comp(n)$.
For the order preserving map
$\Des : Q^{*}_{\Delta} \longrightarrow \Delta$
the poset fiber $\Des^{-1}(\vec{c}\,)$
is the (poset) direct sum of $\beta_{n}^{*}(\vec{c}\,)$
copies of the poset
$\link_{\vec{c}\,}(\Delta)
=
\{\vec{d} \in \Delta \: : \: \vec{d} \leq^{*} \vec{c}\,\}$.
\end{lemma}

\begin{proof}
Let $\sigma$
be an ordered set partition and assume that
the $i$th block $C_{i}$ is the disjoint union
of the two non-empty sets $X$ and $Y$ such that
$\max(X) < \min(Y)$.
Observe now that the two ordered set partitions~$\sigma$
and
$(\ldots, C_{i-1}, X, Y, C_{i+1}, \ldots)$
have the same descent composition,
since there is no descent 
between blocks~$X$ and~$Y$.

Let $\vec{c}$ be a composition in the filter $\Delta$.
For any ordered set partition $\tau$
in the fiber $\Des^{-1}(\vec{c}\,)$
we know that $\tau$ has descent composition $\vec{c}$,
that is, $\Des(\tau)=\vec{c}$.
As $\tau$ can only have descents between blocks,
we know the minimal elements of
$\Des^{-1}(\vec{c}\,)$ have the form $\sigma(\alpha,\vec{c}\,)$
for $\alpha\in\SSSS_{n}$ satisfying $\Des(\alpha)=\vec{c}$
and $\alpha_{n}=n$.
To remain in the same
fiber as these minimal elements,
we are free to break blocks
as in previous paragraph,
hence
$$ \Des^{-1}(\vec{c}\,)
     =
  \{\sigma(\alpha,\vec{d}\,)
      \: : \:
  \vec{c} \leq^{*} \vec{d},
  \:      \vec{d} \in \Delta,
  \:      \Des(\alpha)=\vec{c},\,\alpha_{n}=n\}
  .  $$

Notice that the poset $\link_{\vec{c}\,}(\Delta)$
is isomorphic to
the poset
$\{\sigma(\alpha,\vec{d}\,)\: : \: \vec{d}\leq\vec{c}, \, \vec{d} \in \Delta\}$
for a fixed permutation $\alpha\in\SSSS_{n}$
satisfying $\Des(\alpha)=\vec{c}$ and $\alpha_{n}=n$.
Finally, for a composition
$\vec{e} \in \link_{\vec{c}\,}(\Delta)$
and a permutation $\beta \in \SSSS_{n}$
different from $\alpha$
such that
$\Des(\beta)=\vec{c}$ and $\beta_{n}=n$,
consider the two ordered set partitions
$\sigma(\beta,\vec{e}\,)$ and $\sigma(\alpha,\vec{d}\,)$,
where $\vec{d} \in \link_{\vec{c}\,}(\Delta)$.
By examining the first increasing run in the permutations
$\alpha$ and $\beta$ where their elements differ, we conclude
that the two ordered set partitions
$\sigma(\beta,\vec{e}\,)$ and $\sigma(\alpha,\vec{d}\,)$
are incomparable.
Thus the fiber~$\Des^{-1}(\vec{c}\,)$ is a direct sum of
copies of the poset $\link_{\vec{c}\,}(\Delta)$,
one for each permutation $\alpha$ in $\SSSS_{n}$ satisfying
$\Des(\alpha)=\vec{c}$ and $\alpha(n)=n$.
\end{proof}

\begin{theorem}
Let $\Delta$ be a simplicial complex of compositions
such that every link $\link_{\vec{c}\,}(\Delta)$ has a
Morse matching where the critical cells are facets
of the link $\link_{\vec{c}\,}(\Delta)$. Then the simplicial
complex $Q_{\Delta}^{*}$ has a Morse matching, where 
the number of $i$-dimensional critical cells is given by
equation~\eqref{equation_i_spheres}.
\label{theorem_wedge_of_spheres}
\end{theorem}
\begin{proof}
Apply the
Patchwork Theorem~\cite[Theorem~11.10]{Kozlov_book}
to the poset map
$\Des:Q_{\Delta}^{*} \longrightarrow \Delta$.
By Lemma~\ref{lemma_patchwork}, 
each fiber is a direct sum of
links of $\Delta$,
each of which has a Morse matching.
Each critical cell is a facet.
Hence $Q_{\Delta}^{*}$ is homotopy equivalent to a wedge of spheres,
and thus the order complex $\triangle(\Pistar_{\Delta}-\{\ho\})$ is 
also a wedge of spheres.
The number of $i$-dimensional critical cells 
of $Q_{\Delta}^{*}$ in the fiber $\Des^{-1}(\vec{c}\,)$
is the number of critical cells of dimension $i-|\vec{c}\,|+1$
in the link $\link_{\vec{c}\,}(\Delta)$ times the number of
copies of the link, that is $\beta^{*}_{n}(\vec{c}\,)$.
By summing over all compositions $\vec{c}$ in $\Delta$
the result follows.
\end{proof}

Now suppose that $\Delta$ is a non-pure shellable complex
in the sense of~\cite{Bjorner_Wachs_non_pure_I}.
Then each link 
in $\Delta$ is also shellable, and thus for each link there exists
a discrete Morse matching whose critical
cells are facets of the link;
see Chapter~12 of~\cite{Kozlov_book}.

\begin{corollary}
\label{shellable_order_complex}
If $\Delta$ is a non-pure shellable complex
then Theorem~\ref{theorem_wedge_of_spheres}
applies and the simplicial complex~$Q_{\Delta}^{*}$ has
a Morse matching where the number of $i$-dimensional
critical cells is given by equation~\eqref{equation_i_spheres}.
\end{corollary}
\begin{proof}
This follows directly from two observations:
(i)
a non-pure shellable complex has a Morse matching with all critical cells being facets,
(ii)
each link of a non-pure shellable complex
is non-pure shellable.
See Section~12.1 in~\cite{Kozlov_book}.
\end{proof}

\section{Examples}
\label{section_examples}

In this section we use
Theorem~\ref{theorem_main_theorem_order_complex} and its consequences from Section~\ref{section_consequences}
to derive results about various filters~$\Pistar_{\Delta}$. 
\begin{example}
\label{example_ball}
{\rm
Let $\vec{d}$ be a composition of $n$
into $k+2$ parts and let $\Delta$ be
the simplex generated by~$\vec{d}$.
Since the simplex is homeomorphic to a $k$-dimensional ball, by Corollary~\ref{corollary_ball}
we have that the $k$th reduced homology group
is given by
\begin{align*}
\rH_{k}(\triangle(\Pistar_{\Delta}-\{\ho\}))
&
\cong_{\SSSS_{n-1}}
S^{B^{*}(\vec{d}\,)} ,
\end{align*}
since the only face of $\Delta$ in
the interior of $\Delta$ is the facet $\vec{d}$.
This example illustrates Theorems~5.3
and~7.4
in~\cite{Ehrenborg_Jung}.
Moreover, 
this is the base case of the authors' proof of 
Theorem~\ref{theorem_main_result} using 
the Mayer--Vietoris sequence.
}
\end{example}

\begin{example}
\label{example_sphere}
{\rm
Let $\vec{d}$ be a composition
into $k+3$ parts and let $\Delta$ be
the boundary of the simplex generated by~$\vec{d}$,
that is, $\Delta$ is homeomorphic to a $k$-dimensional
sphere. Then $\Delta$ is shellable
and the order complex $\triangle(\Pistar_{\Delta} -\{\ho\})$
is a wedge of $k$-dimensional spheres.
Now by Corollary~\ref{corollary_sphere}
we have that the $k$th reduced homology group
is given by
$$
\rH_{k}(\triangle(\Pistar_{\Delta}-\{\ho\}))
\cong_{\SSSS_{n-1}}
\bigoplus_{\vec{c} <^{*} \vec{d}}  S^{B^{*}(\vec{c}\,)}
\cong_{\SSSS_{n-1}}
M^{B^{\#}(\vec{d}\,)}
/
S^{B^{*}(\vec{d}\,)} .
$$
Note that we have used
Lemma~\ref{lemma_permutation_module_isomorphism}
to express the permutation module $M^{B^{\#}(\vec{d}\,)}$
as a direct sum of Specht modules.
}
\end{example}

\begin{example}
\label{example_k_r}
{\rm
Let $\vec{d}$ be a composition of $n$ into $k+r$ parts, where $r\geq1$.
Let $\Delta$ be the $k$-skeleton of the simplex generated
by the composition~$\vec{d}$.
Note that $\Delta$ is shellable, so by
Corollary~\ref{shellable_order_complex} 
the order complex $\triangle(\Pistar_{\Delta}-\{\ho\})$ is a wedge of
$k$-dimensional spheres. By
Proposition~\ref{proposition_pure_shellable} we have 
the following calculation in the representation ring:
\begin{align}
\rH_{k}(\triangle(\Pistar_{\Delta}-\{\ho\}))
& \cong_{\SSSS_{n-1}}
\bigoplus_{\substack{\vec{c}\,<\vec{d}\\ |\vec{c}\,|\leq k+2}}
\binom{k+r-|\vec{c}\,|-1}{k-|\vec{c}\,|+2}\cdot S^{B^{*}(\vec{c}\,)} .
\label{equation_k_r}
\end{align}
Here we have used that
$\rbeta_{k-|\vec{c}\,|+1}(\link_{\vec{c}\,}(\Delta))
= (-1)^{k-|\vec{c}\,|+1} \cdot \rchi(\link_{\vec{c}\,}(\Delta))$
since $\link_{\vec{c}\,}(\Delta)$ is shellable.
Lastly, we also used a basic identity on
the alternating sum of binomial coefficients,
which arises in computing the Euler characteristic of the link.
}
\end{example}

\begin{example}
[The $d$-divisible partition lattice with minimal elements removed]
\label{example_boundary_simplex}
{\rm
Let $n$ be a multiple of $d$. Consider the boundary
of the simplex
generated by the composition~$(d,d,\ldots,d)$ of~$n$.
Then~$\Delta$ is a $(n/d-3)$-dimensional simplicial complex,
and $\Pistar_{\Delta}$ is the $d$-divisible partition lattice
without its minimal elements.
By applying Example~\ref{example_sphere}
we obtain that
$\triangle(\Pistar_{\Delta} - \{\ho\})$
is a wedge of $(n/d-3)$-dimensional spheres
and
the reduced homology group is given by
$\rH_{n/d-3}(\Pistar_{\Delta}-\{\ho\})
\cong_{\SSSS_{n-1}}
M^{B^{\#}(d,\ldots,d,d)}
/
S^{B^{*}(d,d,\ldots,d)}$.
}
\end{example}
Setting $d=1$ in the last example shows that the action of
$\SSSS_{n-1}$ on the reduced homology group of
$\triangle(\Pi_{n} - \{\hz,\ho\})$ is
$M^{B^{\#}(1,\ldots,1,1)} = M^{B(1,\ldots,1)}$,
which is the regular representation
of~$\SSSS_{n-1}$.

\begin{example}[The truncated $d$-divisible partition lattice]
{\rm
To generalize Example~\ref{example_boundary_simplex}
and specialize Example~\ref{example_k_r},
let $n = (k+r) \cdot d$
and consider the $k$-skeleton of the simplex generated by
the composition $(d,d, \ldots, d)$ of~$n$.
Here $\Pistar_{\Delta}$ 
consists of all set partitions in the $d$-divisible partition lattice
with at most $k+2$ parts.
Directly we have that
the order complex
$\triangle(\Pistar_{\Delta} - \{\ho\})$
is a wedge of $k$-dimensional spheres
and
its $k$-dimensional reduced homology is
given by equation~\eqref{equation_k_r}.
}
\end{example}

\begin{example}
{\rm
An integer partition
$\lambda = 
\{\lambda_{1}^{m_{1}}, \lambda_{2}^{m_{2}}, \ldots, \lambda_{p}^{m_{p}}\}$
of the non-negative integer~$n$ is called a {\em knapsack partition}
if all the sums
$\sum_{i=1}^{p} e_{i} \cdot \lambda_{i}$,
where $0 \leq e_{i} \leq m_{i}$, are distinct.
In other words, $\lambda$ is a knapsack partition if
$$
\left|
\left\{\sum_{i=1}^{p} j_{i}\cdot\lambda_{i}
 : 0 \leq j_{i} \leq m_{i}\right\}
\right|
=
\prod_{i=1}^{p} (m_{i}+1) .
$$
For a knapsack partition~$\lambda$ into $k-1$ parts
of $n-m$, where $m < n$,
define~$\Delta_{\lambda,m}$
to be the simplicial complex~$\Delta_{\lambda,m}$
which has the facets
$(c_{1}, c_{2}, \ldots, c_{k-1}, c_{k})$
with $\type(c_{1}, c_{2}, \ldots, c_{k-1}) = \lambda$
and the last part~$c_{k}$ is~$m$.
The complex~$\Delta_{\lambda,m}$
is homeomorphic to a $(k-2)$-dimensional ball;
see the proof of Theorem~4.4 in~\cite{Ehrenborg_Readdy_I}.
Applying Corollary~\ref{corollary_ball},
we obtain the following result:
$$ \rH_{k-2}(\triangle(\Pistar_{\Delta_{\lambda,m}}-\{\ho\}))
\cong_{\SSSS_{n-1}} 
\bigoplus_{\vec{c} \in \Int(\Delta_{\lambda,m})} S^{B^{*}(\vec{c}\,)} . $$
Furthermore,
the set of interior faces of $\Delta_{\lambda,m}$
is given by compositions
$\vec{c}$ in~$\Delta_{\lambda,m}$
such that when each part of $\vec{c}$ is written as a sum of parts of
$\lambda$, those parts are distinct.
This example is Theorem~10.3
in~\cite{Ehrenborg_Jung}. Moreover, $\Delta_{\lambda,m}$ 
is shellable, so Theorem~\ref{theorem_wedge_of_spheres}
yields a Morse matching
of $Q_{\Delta_{\lambda,m}}^*$;
see Theorem~8.2 of~\cite{Ehrenborg_Jung}.
}
\end{example}

\section{The Frobenius complex}
\label{section_Frobenius}

\newcommand{\TableOne}{
\begin{table}
\caption{The reduced homology groups of the
order complex $\triangle(\Pi_{n}^{\langle 3,5,7 \rangle} - \{\ho\})$
for the even cases $n = 8$, $10$, $12$ and $14$.
Instead of writing out the notation
$S^{B^{*}(\vec{d},\vec{b}\,)}$ for the Specht modules,
we have drawn the associated border shapes.
Observe that when a row has three boxes,
there is overlap with the row above.
}
\label{table_Pi_3_5_7_even}
\begin{tabular}{c|c|c} 
$n$ & $\rH_{0}$ & $\rH_{2}$ \\ 
\hline
8
&
$\Yboxdim{5pt}
\young(::\hfil\hfil\hfil\hfil,\hfil\hfil\hfil)
\oplus
\young(:::::\hfil\hfil,\hfil\hfil\hfil\hfil\hfil)
$
&
0 \\
\hline
10
&
$\Yboxdim{5pt}
\young(::\hfil\hfil\hfil\hfil\hfil\hfil,\hfil\hfil\hfil)
\oplus
\young(:::::\hfil\hfil\hfil\hfil,\hfil\hfil\hfil\hfil\hfil)
\oplus
\young(:::::::\hfil\hfil,\hfil\hfil\hfil\hfil\hfil\hfil\hfil)
$
&
0
\\
\hline
\vphantom{\rule{0 mm}{8 mm}}
12
\vphantom{\rule{0 mm}{8 mm}}
&
$\Yboxdim{5pt}
\young(:::::\hfil\hfil\hfil\hfil\hfil\hfil,\hfil\hfil\hfil\hfil\hfil)
\oplus
\young(:::::::\hfil\hfil\hfil\hfil,\hfil\hfil\hfil\hfil\hfil\hfil\hfil)
$
&
$\Yboxdim{5pt}
\young(::::::\hfil\hfil,::::\hfil\hfil\hfil,::\hfil\hfil\hfil,\hfil\hfil\hfil)
$
\\
\hline
\vphantom{\rule{0 mm}{8 mm}}
14
\vphantom{\rule{0 mm}{8 mm}}
&
$\Yboxdim{5pt}
\young(:::::::\hfil\hfil\hfil\hfil\hfil\hfil,\hfil\hfil\hfil\hfil\hfil\hfil\hfil)
$
&
$\Yboxdim{5pt}
\young(::::::\hfil\hfil\hfil\hfil,::::\hfil\hfil\hfil,::\hfil\hfil\hfil,\hfil\hfil\hfil)
\oplus
\young(:::::::::\hfil\hfil,::::\hfil\hfil\hfil\hfil\hfil,::\hfil\hfil\hfil,\hfil\hfil\hfil)
$
\\
&
&
$\Yboxdim{5pt}
\oplus\,
\young(:::::::::\hfil\hfil,:::::::\hfil\hfil\hfil,::\hfil\hfil\hfil\hfil\hfil,\hfil\hfil\hfil)
\oplus
\young(:::::::::\hfil\hfil,:::::::\hfil\hfil\hfil,:::::\hfil\hfil\hfil,\hfil\hfil\hfil\hfil\hfil)
$
\end{tabular}
\end{table}
}

\newcommand{\TableTwo}{
\begin{table}
\caption{The reduced homology groups of the
order complex $\triangle(\Pi_{n}^{\langle 3,5,7 \rangle} - \{\ho\})$
for the odd cases $n = 9$, $11$, $13$ and $15$.}
\label{table_Pi_3_5_7_odd}
\begin{tabular}{c|c|c} 
$n$ & $\rH_{1}$ & $\rH_{3}$ \\ 
\hline
9
&
$\Yboxdim{5pt}
\young(::::\hfil\hfil,::\hfil\hfil\hfil,\hfil\hfil\hfil)
$
&
0 \\ 
\hline
11
&
$\Yboxdim{5pt}
\young(::::\hfil\hfil\hfil\hfil,::\hfil\hfil\hfil,\hfil\hfil\hfil)
\oplus
\young(:::::::\hfil\hfil,::\hfil\hfil\hfil\hfil\hfil,\hfil\hfil\hfil)
\oplus
\young(:::::::\hfil\hfil,:::::\hfil\hfil\hfil,\hfil\hfil\hfil\hfil\hfil)
$
&
0 \\
\hline 
13
&
$\Yboxdim{5pt}
\young(:::::::\hfil\hfil\hfil\hfil,::\hfil\hfil\hfil\hfil\hfil,\hfil\hfil\hfil)
\oplus
\young(:::::::\hfil\hfil\hfil\hfil,:::::\hfil\hfil\hfil,\hfil\hfil\hfil\hfil\hfil)
\oplus
\young(::::::::::\hfil\hfil,:::::\hfil\hfil\hfil\hfil\hfil,\hfil\hfil\hfil\hfil\hfil)
$
&
0 \\ 
&
$\Yboxdim{5pt}
\oplus\,
\young(::::\hfil\hfil\hfil\hfil\hfil\hfil,::\hfil\hfil\hfil,\hfil\hfil\hfil)
\oplus
\young(:::::::::\hfil\hfil,::\hfil\hfil\hfil\hfil\hfil\hfil\hfil,\hfil\hfil\hfil)
\oplus
\young(:::::::::\hfil\hfil,:::::::\hfil\hfil\hfil,\hfil\hfil\hfil\hfil\hfil\hfil\hfil)
$
&
0 \\ 
\hline
&
$\Yboxdim{5pt}
\young(:::::::\hfil\hfil\hfil\hfil\hfil\hfil,::\hfil\hfil\hfil\hfil\hfil,\hfil\hfil\hfil)
\oplus
\young(:::::::::\hfil\hfil\hfil\hfil,::\hfil\hfil\hfil\hfil\hfil\hfil\hfil,\hfil\hfil\hfil)
\oplus
\young(:::::::\hfil\hfil\hfil\hfil\hfil\hfil,:::::\hfil\hfil\hfil,\hfil\hfil\hfil\hfil\hfil)
$
&
\\
15
&
$\Yboxdim{5pt}
\oplus\,
\young(::::::::::::\hfil\hfil,:::::\hfil\hfil\hfil\hfil\hfil\hfil\hfil,\hfil\hfil\hfil\hfil\hfil)
\oplus
\young(:::::::::\hfil\hfil\hfil\hfil,:::::::\hfil\hfil\hfil,\hfil\hfil\hfil\hfil\hfil\hfil\hfil)
\oplus
\young(::::::::::::\hfil\hfil,:::::::\hfil\hfil\hfil\hfil\hfil,\hfil\hfil\hfil\hfil\hfil\hfil\hfil)
$
&
$\Yboxdim{5pt}
\young(::::::::\hfil\hfil,::::::\hfil\hfil\hfil,::::\hfil\hfil\hfil,::\hfil\hfil\hfil,\hfil\hfil\hfil)
$
\\
&
$\Yboxdim{5pt}
\oplus\,
\young(::::::::::\hfil\hfil\hfil\hfil,:::::\hfil\hfil\hfil\hfil\hfil,\hfil\hfil\hfil\hfil\hfil)
$
&
\end{tabular}
\end{table}
}

We now consider a different class of examples stemming
from~\cite{Clark_Ehrenborg}.
Let~$\Lambda$ be a semigroup of positive integers,
that is, a subset of the positive integers which is closed
under addition.
Let~$\Delta_{n}$ be the collection of all compositions of $n$
whose parts belong to $\Lambda$,
that is,
$$
  \Delta_{n}
    =
  \{ (c_{1}, \ldots, c_{k}) \in \Comp(n) \:\: : \:\: 
          c_{1}, \ldots, c_{k} \in \Lambda   \}  .
$$
Since $\Lambda$ is closed under addition,
we obtain that $\Delta_{n}$ is a filter in the poset of
compositions~$\Comp(n)$ and hence we view it as
a simplicial complex.
This complex is known as the Frobenius complex;
see~\cite{Clark_Ehrenborg}.
Moreover, since $\Lambda$ is a semigroup,
the collection of integer partitions of $n$
with parts in~$\Lambda$ is a filter, therefore,
using
Lemma~\ref{lemma_integer_filter}
the associated filter in
the partition lattice is given by
$$   \Pi_{n}^{\Lambda}
     =
        \{ \{B_{1}, \ldots, B_{k}\} \in \Pi_{n}
     \:\: : \:\:   |B_{1}|, \ldots, |B_{k}|  \in \Lambda\}  .   $$

Let $\Psi_{n}$ be the generating function
$$ \Psi_{n}
   =
    \sum_{i \geq -1} \rbeta_{i}(\Delta_{n}) \cdot t^{i+1} . $$
Observe that for a composition $\vec{c}$ in $\Delta_{n}$
we have that
the link $\link_{\vec{c}}(\Delta_{n})$ is given by
the join
\begin{equation}
\link_{\vec{c}}(\Delta_{n})
=
\Delta_{c_{1}} * \Delta_{c_{2}} * \cdots * \Delta_{c_{k}} .
\label{equation_Frobenius_link}
\end{equation}
We can apply the K\"unneth theorem to obtain
that the $i$th reduced Betti number
of the link is given by
$$
    \rbeta_{i}(\link_{\vec{c}\,}(\Delta_{n}))
  =
    [t^{i+1}] \Psi_{c_{1}} \cdot \Psi_{c_{2}} \cdots \Psi_{c_{k}}  .
$$
Using Theorem~\ref{theorem_main_theorem_order_complex},
the $i$th reduced Betti number of
the order complex $\triangle(\Pi_{n}^{\Lambda} - \{\ho\})$
is given in the representation ring of $\SSSS_{n-1}$ by
\begin{align*}
\rH_{i}(\triangle(\Pi_{n}^{\Lambda} - \{\ho\}))
\cong_{\SSSS_{n-1}}
\sum_{\vec{c}} [t^{i_{\vec{c}\,}+1}]
\Psi_{c_{1}} \cdot
\Psi_{c_{2}} \cdots
\Psi_{c_{k}}
\cdot
S^{B^{*}(\vec{c}\,)}  ,
\end{align*}
where the sum is over all compositions
$\vec{c} = (c_{1}, c_{2}, \ldots, c_{k})$ of $n$.

A more explicit approach is possible when the
complex $\Delta_{n}$ has a discrete Morse matching.
By combining equation~\eqref{equation_Frobenius_link},
Lemma~\ref{lemma_join_critical_cells},
and a Morse matching from~\cite{Clark_Ehrenborg}, 
we create a Morse matching on every link.
We will see this method in the remainder of this section.

We continue by studying one concrete example.
Let $a$ and $d$ be two positive integers.
Let~$\Lambda$ be the semigroup
generated by the arithmetic progression
\begin{align*}
\Lambda
& = 
\langle a, a+d, a+2d, \ldots \rangle . \\
\intertext{Since for $j \geq a$ we have that
$a + j \cdot d = d \cdot a + a + (j-a) \cdot d$,
the semigroup is generated by
the finite arithmetic progression
} 
\Lambda
& = 
\langle a, a+d, a+2d, \ldots, a+(a-1)d\, \rangle .
\end{align*}
Clark and Ehrenborg proved that the Frobenius complex $\Delta_{n}$
is a wedge of spheres of different dimensions;
see~\cite[Theorem~5.1]{Clark_Ehrenborg}.
Observe that their result
is formulated in terms of sets,
instead of compositions.
However,
the two notions are equivalent via the
natural bijection given by 
sending 
a composition $(c_{1}, c_{2}, \ldots, c_{k})$ of~$n$
to the subset
$\{c_{1},c_{1}+c_{2}, \ldots, c_{1} \plusdots c_{k-1}\}$
of the set~$[n-1]$.
To state their result, let $A$ be the set
$\{a+d, a+2d, \ldots, a+(a-1) \cdot d\}$.

\begin{proposition}
For $n$ in the semigroup $\Lambda$,
there is a discrete Morse matching on the
Frobenius complex $\Delta_{n}$ such that
the critical cells are compositions $\vec{c} = (c_{1}, \ldots, c_{k})$
characterized by
\begin{itemize}
\item[(i)]
All but the last entry of the composition
belongs to the set $A$,
that is,
$c_{1}, \ldots, c_{k-1} \in A$.
\item[(ii)]
The last entry $c_{k}$ belongs to 
$\{a\} \cup A$.
\end{itemize}
Furthermore, all the critical cells are facets.
\label{proposition_Clark_Ehrenborg}
\end{proposition}
\begin{proof}
When $a$ and $d$ are relative prime,
that is, $\gcd(a,d) = 1$, this result is
Lemma~5.10
in~\cite{Clark_Ehrenborg}.
When $a$ and $d$ are not relative prime,
the result follows by scaling down
the three parameters~$a$, $d$ and~$n$ by
$a^{\prime} = a/\gcd(a,d)$,
$d^{\prime} = d/\gcd(a,d)$
and
$n^{\prime} = n/\gcd(a,d)$.
Now the result applies the semigroup
$\Lambda^{\prime}
= 
\langle a^{\prime}, a^{\prime}+d^{\prime}, a^{\prime}+2d^{\prime}, 
  \ldots\rangle$
and its associated Frobenius complex~$\Delta^{\prime}_{n^{\prime}}$.
However, this complex is isomorphic to $\Delta_{n}$ by
sending the composition
$\vec{c} = (c_{1}, \ldots, c_{k})$
in~$\Delta^{\prime}_{n^{\prime}}$
to the composition
$\gcd(a,d) \cdot \vec{c} = 
(\gcd(a,d) \cdot c_{1}, \ldots, \gcd(a,d) \cdot c_{k})$
in~$\Delta_{n}$.
\end{proof}

\begin{corollary}
The order complex
$\triangle(\Pi_{n}^{\Lambda} - \{\ho\})$ is
a wedge of spheres.
\end{corollary}
\begin{proof}
Since $\Delta_{n}$ has a discrete Morse matching
where each critical cell is a facet, $\Delta_{n}$ is
homotopy equivalent to a wedge of spheres.
Furthermore, by equation~\eqref{equation_Frobenius_link}
we know that every link of $\Delta_{n}$ is a wedge of
spheres. Finally, by
Corollary~\ref{corollary_every_link_wedge_spheres}
we obtain the result.
\end{proof}

Next we need to extend
Lemma~\ref{lemma_permutation_module_isomorphism}
to collect Specht modules together.
We call the sum $c_{1} + c_{2} \plusdots c_{j}$
an {\em initial sum} of a composition $\vec{c} = (c_{1}, c_{2}, \ldots, c_{k})$
for $1 \leq j \leq k$.
\begin{definition}
\label{definition_d_b}
For an interval
$[\vec{d}, \vec{b}\,]$ in the lattice of compositions $\Comp(n)$
let 
$B^{*}(\vec{d},\vec{b}\,)$
be the skew-shape
where the row lengths are given by
$d_{1}, d_{2}, \ldots, d_{r-1}, d_{r}-1$
and
if the initial sum $d_{1} + \cdots + d_{j}$
is equal to an initial sum of the composition $\vec{b}-1$,
then $j$th row and the $(j+1)$st row overlap in one column. All other rows 
of $B^{*}(\vec{d},\vec{b}\,)$ are non--overlapping.
\end{definition}
As an example, if $\vec{d}=(2,5,4,1,3,2)$ and $\vec{b}=(2+5+4,1+3,2)$, 
then $B^*(\vec{d},\vec{b}\,)$ is the border strip with row lengths
$2,5,4,1,3$ and $2-1=1$ which overlaps between the rows of length $4$ and $1$ 
and the rows of length $3$ and~$1$.  
Note that
$B^{*}(\vec{d},(n)) = A^{*}(\vec{d}\,)$.

The proof of the next lemma 
is the same as the proof of
Lemma~\ref{lemma_permutation_module_isomorphism},
that is, it uses 
jeu-de-taquin
moves where two adjacent rows do not overlap.
\begin{lemma}
\label{lemma_interval_in_composition_poset}
Let $\vec{b}$ and $\vec{d}$ be two compositions
in $\Comp(n)$ such that $\vec{d} \leq \vec{b}$.
Then Specht module~$S^{B^{*}(\vec{d},\vec{b})}$
is given by the direct sum
$$  
S^{B^{*}(\vec{d},\vec{b})}
\cong_{\SSSS_{n-1}}
\bigoplus_{\vec{d} \leq \vec{c} \leq \vec{b}}   S^{B^{*}(\vec{c}\,)} .
$$
\end{lemma}

In order to state the main result for the semigroup
$\Lambda = \langle a, a + d, a + 2 d, \ldots \rangle$
and the associated filter in the partition lattice,
we need one last definition.
\begin{definition}
For a composition~$\vec{d}$ of $n$ with entries in the set
$\{a\} \cup A$
let $\vec{b}(\vec{d}\,)$ be the composition
greater than or equal to~$\vec{d}$ 
obtained by adding runs of entries of $\vec{d}$
together where each run ends with the entry $a$.
\end{definition}
As an example, for
$a = 3$, $d = 2$ we have $A = \{5,7\}$.
Hence for the composition
$\vec{d} = (5,3,    7,5,3,  3,   7,5)$
we obtain
$\vec{b}(\vec{d}\,)
=    (5+3,    7+5+3,  3,   7+5) 
=    (8,        15,        3,   12)$. 

\begin{remark}
{\rm
Observe that the skew-shape
$S^{B^{*}(\vec{d}, \vec{b}(\vec{d}\,))}$
has the row lengths $d_{1}, \ldots, d_{r-1}, d_{r}-1$
and satisfies the condition that
$d_{i} = a$ if and only if there is overlap between
$i$th and $(i+1)$st rows. See Definition~\ref{definition_d_b}.
}
\label{remark_a_is_special}
\end{remark}

\TableOne

\begin{theorem}
Let $a$ and $d$ be two positive integers
and let $\Pi_{n}^{\Lambda}$ be the
filter in the partition lattice~$\Pi_{n}$
where each partition $\pi$ consists of blocks
whose cardinalities belong to the semigroup~$\Lambda$
generated by the arithmetic progression
$a, a + d, \ldots, a + (a-1) \cdot d$.
Then the $i$th reduced homology group of
the order complex $\triangle(\Pi_{n}^{\Lambda} - \{\ho\})$
is given by the direct sum
\begin{align*}
   \rH_{i}(\triangle(\Pi_{n}^{\Lambda} - \{\ho\}))
& \cong_{\SSSS_{n-1}}
       \bigoplus_{\vec{d}}
                                 S^{B^{*}(\vec{d}, \vec{b}(\vec{d}\,))} ,
\end{align*}
where the sum is over all compositions~$\vec{d}$ into $i+2$
parts such that every entry belongs to the set
$\{a\} \cup A = \{a, a+d, a + 2 \cdot d, a + (a-1) \cdot d\}$.
\label{theorem_a_d}
\end{theorem}
\begin{proof}
Let $\vec{c}$ be a composition in the complex $\Delta_{n}$.
Using the Morse matching
given by 
Proposition~\ref{proposition_Clark_Ehrenborg}
and Lemma~\ref{lemma_critical_cell_join},
we obtain that a critical cell $\vec{d}$
in the link
$\link_{\vec{c}}(\Delta_{n})
=
\Delta_{c_{1}} * \Delta_{c_{2}} * \cdots * \Delta_{c_{k}}$
is a composition
$\vec{d} \leq \vec{c}$
where the entries of $\vec{d}$ belong to the set $\{a\} \cup A$.
Furthermore,
in the run of entries of $\vec{d}$ that sums to the entry~$c_{i}$
of the composition~$\vec{c}$, only the last entry of the run
is allowed to be equal to $a$.
Using Theorem~\ref{theorem_main_theorem_order_complex}
we have
\begin{align*}
   \rH_{i}(\triangle(\Pi_{n}^{\Lambda} - \{\ho\}))
 & \cong_{\SSSS_{n-1}}
       \bigoplus_{\vec{c} \in \Delta_{n}}
           \rH_{i-|\vec{c}\,|+1}(\link_{\vec{c}\,}(\Delta_{n}))
                \otimes
                                 S^{B^{*}(\vec{c}\,)}.  \\
 & \cong_{\SSSS_{n-1}}
       \bigoplus_{\vec{c} \in \Delta_{n}}
       \bigoplus_{\vec{d}}
                                 S^{B^{*}(\vec{c}\,)}, 
\end{align*}
where the inner sum consists of critical compositions~$\vec{d}$
satisfying the conditions discussed in the previous paragraph
and with $|\vec{d}\,| = i+2$.
By changing the order of summation
we obtain
\begin{align*}
   \rH_{i}(\triangle(\Pi_{n}^{\Lambda} - \{\ho\}))
 & \cong_{\SSSS_{n-1}}
       \bigoplus_{\vec{d}}
       \bigoplus_{\vec{c}} S^{B^{*}(\vec{c}\,)} ,
\end{align*}
where
the outer direct sum is over all compositions~$\vec{d}$
of $n$ into $i+2$ parts where each part is in the set $\{a\} \cup A$
and
the inner direct sum is over all compositions~$\vec{c}$
greater than $\vec{d}$ obtained by 
adding runs of entries of $\vec{d}$ where an entry equal to $a$
can only be at the end of a run.
The inner direct sum is hence given by
the Specht module~$S^{B^{*}(\vec{d},\vec{b}(\vec{d}))}$
by
Remark~\ref{remark_a_is_special}
and
Lemma~\ref{lemma_interval_in_composition_poset},
and therefore the result follows.
\end{proof}

\TableTwo

\begin{corollary}
The order complex $\triangle(\Pi_{n}^{\Lambda} - \{\ho\})$
only has non-vanishing reduced homology
in dimension $i$ when
$n \equiv (i+2) \cdot a \bmod d$.
\end{corollary}
\begin{proof}
Since all entries in the set $\{a\} \cup A$ are
congruent to $a$ modulo $d$, we have
$n = \sum_{j=1}^{i+2} d_{j} \equiv (i+2) \cdot a \bmod d$.
\end{proof}

In Tables~\ref{table_Pi_3_5_7_even}
and~\ref{table_Pi_3_5_7_odd}
we have explicitly calculated the
reduced homology groups
for the order complex
$\triangle(\Pi^{\langle 3,5,7 \rangle}_{n} - \{\ho\})$
for $8 \leq n \leq 15$,
that is, the case $a=3$ and $d=2$.
In this case the previous corollary
implies that the order complex only
has non-vanishing homology in dimensions
of the same parity as~$n$.

\begin{example}
{\rm
When the integer $d$ divides the integer $a$,
the homology groups of $\Pi_{n}^{\Lambda}$ have been studied.
In this case, the filter $\Pi_{n}^{\Lambda}$ consists of all partitions
where the block sizes are divisible by~$d$
and the block sizes are greater than or equal to $a$.
This filter was studied by Browdy~\cite{Browdy},
and our Theorem~\ref{theorem_a_d} reduces to
Browdy's result; see Corollary~5.3.3 in~\cite{Browdy}.
}
\end{example}

\begin{example}
{\rm
The previous example is particularly nice when $d=1$.
The semigroup~$\Lambda$ is given by
$\Lambda = \{n \in \mathbb{P} \: : \: n \geq a\}$
and the filter $\Pi_{n}^{\Lambda}$
consists of all partitions where
$1, 2, \ldots, a-1$ are forbidden block sizes.
In this case it follows by Billera and
Meyers~\cite{Billera_Myers} that $\Delta_{n}$ is non-pure shellable.
Additionally, Bj\"orner and Wachs~\cite{Bjorner_Wachs_non_pure_I}
gave an $EL$-labelling of $\Pi_{n}^{\Lambda} \cup \{\hz\}$.
This order complex was also
considered by Sundaram in Example~4.4 in~\cite{Sundaram_Hopf}.
}
\end{example}

\section{The partition filter $\Pi_{n}^{\langle a,b \rangle}$}
\label{section_partition_filter_a_b}

Let $a$ and $b$ be two relatively prime integers greater than $1$.
Let $\Pi_{n}^{\langle a,b \rangle}$ be the filter in $\Pi_{n}$ generated by all partitions 
whose block sizes are all $a$ or $b$.
As an example, $\Pi_{n}^{\langle 2,3 \rangle}$
consists of all partitions in $\Pi_{n}$ with no singleton blocks.
The corresponding complex $\Delta_{n}$
in $\Comp(n)$ consists of all compositions of~$n$
whose parts are contained in the set
$\langle a,b\rangle=\{i\cdot a+j\cdot b\::\:a,b\in\mathbb{N}\}$.
When $a=2$ and $b=3$ the complex $\Delta_{n}$ is known as
the complex of sparse sets;
see~\cite{Clark_Ehrenborg,Kozlov_sparse_sets}.

Following Theorem~4.1 in~\cite{Clark_Ehrenborg},
we define the set
$A = \{ n\in\Ppp : n\equiv0,a,b\text{ or }a+b \bmod ab\}$
and the function $h : A \longrightarrow \mathbb{Z}_{\geq -1}$
as follows:
\begin{equation}
h(n) =
\begin{cases} \frac{2n}{ab}-2 &\mbox{if } n \equiv 0 \bmod{ab}, \\ 
 \frac{2(n-a)}{ab}-1 &\mbox{if } n \equiv a \bmod{ab}, \\ 
 \frac{2(n-b)}{ab}-1 &\mbox{if } n \equiv b \bmod{ab}, \\ 
 \frac{2(n-a-b)}{ab} &\mbox{if } n \equiv a+b \bmod{ab}.
\end{cases}
\label{equation_h}
\end{equation}
Then Theorem 4.1 in~\cite{Clark_Ehrenborg}
states that $\Delta_{n}$ is either homotopy equivalent to 
a sphere or is contractible, according to 
$$\Delta_{n}\simeq\begin{cases} S^{h(n)} &\mbox{if } n\in A,\\
\mbox{point}&\mbox{otherwise.}
\end{cases}
$$
Using equation~\eqref{equation_Frobenius_link}
we see that if $\vec{c} \in \Delta_{n}$
has any part not in $A$, then $\link_{\vec{c}\,}(\Delta_{n})$ is contractible.
If each part of $\vec{c}$ is in~$A$, then
$\link_{\vec{c}\,}(\Delta_{n})
\simeq S^{h(c_{1})}*\cdots*S^{h(c_{k})}=S^{h(\vec{c}\,)}$,
where we define
$h(\vec{c}\,)=k-1+\sum_{j=1}^{k} h(c_{j})$
for compositions~$\vec{c}$ with all parts in $A$, since the join of 
an $n$-dimensional sphere and an $m$-dimensional sphere
is an $(n+m+1)$-dimensional sphere.
Note that
$h(\vec{c}\,)$ is undefined for all other compositions.

For a composition $\vec{c}=(c_{1}, \ldots, c_{k})$ of $n$
with all of its parts in $A$, 
let $\dim(\vec{c}\,)$ denote the dimension of the reduced homology
of $\triangle(\Pi_{n}^{\langle a,b \rangle} - \{\ho\})$
to which the composition~$\vec{c}$ contributes.
That is,
$\dim(\vec{c}\,)$ is given by
\begin{equation}
\dim(\vec{c}\,) = h(\vec{c}\,) + k-1 = \sum_{i=1}^{k} h(c_{i}) + 2k-2.
\label{equation_composition_dimension}
\end{equation}
We can apply
Theorem~\ref{theorem_main_theorem_order_complex}
to obtain

\begin{theorem}
Let $2\leq a<b$ with $\gcd(a,b)=1$.
Then the $i$th reduced homology group of
$\triangle(\Pi_{n}^{\langle a,b \rangle}-\{\ho\})$
is given by the direct sum of Specht modules
$\bigoplus_{\vec{c}\,\in F_{i}} S^{B^{*}(\vec{c}\,)}$,
where $F_{i}$ is the collection of compositions~$\vec{c}$ of $n$
where all the parts are in the set $A$ with
$\dim(\vec{c}\,) = i$.
\end{theorem}
\begin{proof}
We directly have
\begin{align*}
\rH_{i}(\Delta(\Pi_{n}^{\langle a,b \rangle} - \{\ho\}))
& \cong_{\SSSS_{n-1}}
\bigoplus_{\vec{c} \in \Delta}
\rH_{i_{\vec{c}}}(\link_{\vec{c}\,}(\Delta_{n})) \otimes S^{B^{*}(\vec{c}\,)} \\
& \cong_{\SSSS_{n-1}}
\bigoplus_{\vec{c}\in F_{i}}
\rH_{i_{\vec{c}\,}}(S^{h(\vec{c}\,)}) \otimes S^{B^{*}(\vec{c}\,)}\\
& \cong_{\SSSS_{n-1}}
\bigoplus_{\vec{c}\in F_{i}} S^{B^{*}(\vec{c}\,)} .
\qedhere
\end{align*}
\end{proof}

We now describe the top and
bottom reduced homology of
the order complex $\Delta(\Pi_{n}^{\langle a,b \rangle} - \{\ho\})$.
We begin with the top homology.
\begin{proposition}
\label{proposition_a_b_top}
Let $2\leq a<b$ with $\gcd(a,b)=1$.
Let $r$ be the unique integer such that
$0 \leq r < a$ and $n\equiv rb\bmod a$.
Then the top homology of $\triangle(\Pi_{n}^{\langle a,b \rangle}-\{\ho\})$, which occurs
in dimension $(n-r(b-a))/a-2$,
is given by the direct sum of Specht modules
$\bigoplus_{\vec{c}\,\in R} S^{B^{*}(\vec{c}\,)}$,
where $R$ is the collection of compositions~$\vec{c}$ of~$n$ where 
exactly $r$ of the parts are equal to $b$ or $a+b$,
and the remaining parts are all equal to~$a$.
\end{proposition}
\begin{proof}
We present two procedures that will
change a composition~$\vec{c}$ into
another composition~$\cp$ such that
the dimension of contribution from~$\cp$ is greater than
the contribution of~$\vec{c}$,
that is, $\dim(\vec{c}\,) < \dim(\cp)$.
The compositions which we cannot
improve with this procedure are those described in the statement of the
proposition.

We now describe the first replacement 
procedure.
If the composition~$\vec{c}$ has a part of the form
\begin{itemize}
\item[(i)] $jab$, replace it with $jb$ $a$'s,
\item[(ii)] $jab+a$, replace it with $(jb+1)$ $a$'s,
\item[(iii)] $jab+b$, replace it with $jb$ $a$'s and one $b$,
\item[(iv)] $jab+a+b$, replace it with $(jb+1)$ $a$'s and one $b$,
\end{itemize}
to obtain a new composition~$\cp$.
We claim that $\dim(\cp) - \dim(\vec{c}\,) = (b-a) \cdot j$.
We check the computation in the case (iv),
the other three cases are similar.
The difference 
$\dim(\cp) - \dim(\vec{c}\,)$ only depends on the parts affected
and the number of them.
Hence
\begin{align*}
\dim(\cp) - \dim(\vec{c}\,)
& = 
[(jb+1) \cdot h(a) + h(b) + 2(jb+2)] - [h(jab+a+b) + 2] \\
& = 
[jb+2] - [2j + 2] = (b-a) \cdot j > 0 ,
\end{align*}
using that $h(a) = h(b) = -1$ and $h(jab+a+b) = 2j$.
Hence this procedure increases the dimension.

Iterating this procedure
we obtain
a new composition
with all the parts
of the form $a$, $b$ and~$a+b$.

The second replacement procedure is as follows.
Assume that there are $a$ parts of the composition~$\vec{c}$
that are different from~$a$.
Assume that $p$ of these parts are equal to $a+b$,
and hence $a-p$ of them are equal to $b$.
Replace these $a$ parts with $b+p$ parts equal to $a$
to obtain a new composition~$\cp$.
\begin{align*}
\dim(\cp) - \dim(\vec{c}\,)
& =
[(b+p) \cdot h(a) + 2 (b+p)]
-
[p \cdot h(a+b) + (a-p) \cdot h(b) + 2a] \\
& =
[b+p] - [a+p] = b-a > 0.
\end{align*}
Hence the new composition~$\cp$ contributes
to a homology of dimension
$b-a > 0$ greater than the composition~$\vec{c}$ does.

Iterating the last procedure,
we are left with a composition~$\vec{c}$ where the number of parts
different from $a$ is at most $a-1$.
By considering the equation
$c_{1} + \cdots +c_{k} = n$ modulo~$a$,
we obtain the number of parts
different from $a$ is given by the integer $r$ from the statement of the proposition.
Additionally, switching between one part of $a+b$ and the two parts
$a$ and $b$ does not change the dimension of the contribution of the composition.
Finally, we compute the contribution of the composition
$(\underbrace{a, \ldots, a}_{(n-br)/a}, \underbrace{b, \ldots, b}_{r})$
to obtain the desired dimension.
\end{proof}

\begin{corollary}
Let $2\leq a<b$ with $\gcd(a,b)=1$.
Assume that $n$ is divisible by $a$.
Then the top homology of $\triangle(\Pi_{n}^{\langle a,b \rangle}-\{\ho\})$,
which occurs
in dimension $n/a -2$,
is the Specht module~$S^{B^{*}(a,a, \ldots, a)}$.
\end{corollary}
\begin{proof}
When $a$ divides $n$, then the integer $r$ of Proposition~\ref{proposition_a_b_top}
is $0$. Thus the only contribution to reduced homology in dimension $n/a-2$ is
given by $(a,a, \ldots, a)$.
\end{proof}

We now turn our attention to the bottom reduced homology.

\begin{proposition}
\label{proposition_a_b_bottom}
Let $3\leq a<b$ with $\gcd(a,b)=1$.
Let $r$ and $s$ be the two unique integers
such that 
$$ n \equiv rb \bmod a, 
  \:\:\:\: 0 \leq r < a, 
  \:\:\:\: n \equiv sa \bmod b
  \:\:\:\: \text{ and } \:\:\:\: 0 \leq s < b . $$
Then the bottom reduced homology of $\triangle(\Pi_{n}^{\langle a,b \rangle}-\{\ho\})$
occurs in dimension $2\cdot\frac{n-sa-rb}{ab}+r+s-2$, and 
is given by the direct sum of Specht modules $S^{B^{*}(\vec{c}\,)}$
over all compositions~$\vec{c}$
such that
the number of parts of $\vec{c}$
of the form $j\cdot ab+a$ and $j\cdot ab+a+b$
is~$s$
and
the number of parts of the form $j\cdot ab+b$ and $j\cdot ab+a+b$ 
is~$r$.
\end{proposition}
\begin{proof}
Just as in Proposition~\ref{proposition_a_b_top}, we will define
replacement procedures, where our goal now
is to decrease the dimension of the homology that our composition contributes to, rather than increase it, as was the case in Proposition~\ref{proposition_a_b_top}.

The first procedure takes $b$ parts of the composition~$\vec{c}$
of the form $jab+a$ and $jab+a+b$ and 
subtracts~$a$ from each of these $b$~parts,
and adjoins a new part $ab$.
Notice that 
the resulting new composition $\cp$
remains a composition of $n$.
Observe that
$h(jab) = h(jab+a) -1$,
$h(jab+b) = h(jab+a+b) -1$,
and $h(ab) = 0$.
Hence the dimension $\cp$ contributes to
is
$\dim(\cp) = \sum_{i=1}^{k+1} h(c_{i}^{\prime}) + 2(k+1) - 2
= \sum_{i=1}^{k} h(c_{i}) - b + 2(k+1) - 2
= \dim(\vec{c}\,) - b+2 < \dim(\vec{c}\,)$.

There is one small caveat.
In the procedure, replacing a part $a$ with $0$
we obtain a weak composition, that is, we can introduce zero entries.
Note the natural extension of the function~$h$
satisfies $h(0) = -2$.
Assume that $\cp$ has a zero entry, say in its last entry,
and let $\cpp$ be the (weak) composition with this last entry removed.
Then we have that
$\dim(\cp) = \sum_{i=1}^{k+1} h(c_{i}^{\prime}) + 2(k+1) - 2
= \sum_{i=1}^{k} h(c_{i}^{\prime\prime}) + 2k - 2
= \dim(\cpp)$.
Thus zero entries can be removed without 
changing the dimension.

The second procedure
is symmetric to the first in the two parameters $a$ and~$b$.
That is, it takes $a$~parts of the composition $\vec{c}$
of the form $jab+b$ and $jab+a+b$ and 
subtracts~$b$ from each of these $a$ parts and
adjoins a new part~$ab$. Now we have 
$\dim(\cp) = \dim(\vec{c}\,) - a+2 < \dim(\vec{c}\,)$,
using the fact that $a \geq 3$.

Iterating these two procedures we obtain a composition
which 
has at most $b-1$ parts of the form $jab+a$ and $jab+a+b$,
and
at most $a-1$ parts of the form $jab+b$ and $jab+a+b$.
Hence this composition satisfies the condition of the statement
of the proposition.
Finally, one has to observe that all such composition contribute
to the same dimension.
\end{proof}

\begin{corollary}
Let $3 \leq a < b$, $\gcd(a,b) = 1$
and let $n$ be divisible by~$ab$.
Then the bottom reduced homology
of the order complex
$\triangle(\Pi_{n}^{\langle a,b \rangle} - \{\ho\})$
is given by
the permutation module
$M^{B^{\#}(ab, \ldots, ab, ab)} = M^{B(ab, \ldots, ab, ab-1)}$.
\end{corollary}

\begin{proof}
We have $r=s=0$. Hence the compositions only
have parts of the form $j \cdot ab$. The result follows
from Lemma~\ref{lemma_permutation_module_isomorphism}.
\end{proof}

We end with a complete description in the case when $a=2$.
\begin{proposition}
Let $b$ be odd and greater than or equal to~$3$.
Then the $i$th reduced homology of
$\triangle(\Pi_{n}^{\langle 2,b \rangle}-\{\ho\})$ is given by 
the direct sum 
of Specht modules $S^{B^{*}(\vec{c}\,)}$
over all compositions $\vec{c}$
with all parts congruent to $0$ or $2$ modulo~$b$,
where exactly $(b(i+2) - n)/(b-2)$
entries of $\vec{c}$ are congruent to $2$ modulo~$b$.
The bottom reduced homology occurs in dimension 
$\lceil n/b \rceil - 2$.
Furthermore, when $b$ divides~$n$
the bottom reduced homology is given by
the permutation module $M^{B^{\#}(b,\ldots,b,b)}=M^{B(b,\ldots,b,b-1)}$.
\end{proposition}
\begin{proof}
Since $a=2$
the expression for $h(n)$ in equation~\eqref{equation_h}
reduces to 
$h(n)=\lceil n/b \rceil - 2$
and the set~$A$ reduces to
$\{n \in \Ppp \: : \: n \equiv 0,2 \bmod b\}$.
Let $\vec{c}$ be a composition of $n$ into $k$ parts,
where each part belongs to the set $A$.
Furthermore, assume that $\vec{c}$
has $s$ entries congruent to $2$ modulo~$b$.
The contribution of $\vec{c}$ to the reduced homology of
$\triangle(\Pi_{n}^{\langle 2,b \rangle}-\{\ho\})$,
given by equation~\eqref{equation_composition_dimension}, is in dimension
$$
\dim(\vec{c}\,)
 =
\sum_{i=1}^{k} h(c_{i}) + 2k - 2 
 =
\sum_{i=1}^{k} \left\lceil\frac{c_{i}}{b}\right\rceil - 2 
 =
\frac{\sum_{i=1}^{k} c_{i} + s \cdot (b-2)}{b}  - 2 
 =
\frac{n + s \cdot (b-2)}{b}  - 2 .
$$
Solving for $s$ in this equation yields the desired expression.

For real numbers $x$ and $y$ we have
the inequality
$\lceil x \rceil + \lceil y \rceil \geq \lceil x+y \rceil$.
Hence we obtain the lower bound on the dimension of the homology:
$\dim(\vec{c}\,)
=
\sum_{i=1}^{k} \left\lceil\frac{c_{i}}{b}\right\rceil - 2
\geq
\left\lceil\frac{n}{b}\right\rceil - 2$.
When $b$ divides~$n$ the only way to obtain equality in the previous inequality is when
all the parts of the composition are divisible by~$b$.
The bottom reduced homology group is then the direct sum over
all compositions~$\vec{c}$ of $n$ where each part is divisible by $b$,
that is, $(b,b,\ldots,b) \geq^{*} \vec{c}$.
Hence we obtain the permutation module
$M^{B^{\#}(b,\ldots,b,b)} = M^{B(b,\ldots,b,b-1)}$
by Lemma~\ref{lemma_permutation_module_isomorphism}.
\end{proof}

\section{Concluding remarks}

Using Theorem~\ref{theorem_main_theorem_order_complex}
we have been able to classify the action of $\SSSS_{n-1}$ on 
the top homology of $\triangle(\Pistar_{\Delta}-\{\ho\})$
for any complex $\Delta\subseteq\Comp(n)$.
In the case when $\triangle(\Pistar_{\Delta} - \{\ho\})$ is shellable,
is there an $EL$-labelling of $\Pistar_{\Delta}\cup\{\hz\}$
that realizes this shelling order?

Is there a way we can classify the $\SSSS_{n}$-action
on the homology groups of $\triangle(\Pistar_{\Delta}-\{\ho\})$
rather than the $\SSSS_{n-1}$-action?
Browdy described the matrices
representing the action of $\SSSS_{n}$
on the cohomology groups of the filter
with block sizes belonging
to the arithmetic progression
$k \cdot d, (k+1) \cdot d, \ldots$;
see~\cite[Section~5.4]{Browdy}.

The partition lattice is naturally associated
with the symmetric group, that is, the Coxeter group
of type~$A$.
Miller~\cite{Miller} has extended the results about the
filter $\Pistar_{\vec{c}}$ to other root systems.
Hence it is natural to ask if our results for the filter~$\Pistar_{\Delta}$
can be extended to other root systems.

Is there a non-pure shelling of the
Frobenius complex generated by $a$ and~$b$?
Alternatively, is there a Morse matching for this
Frobenius complex such that all the critical cells
are facets? While we do have this property 
for $\Lambda$ defined by an arithmetic progression as in Section~\ref{section_Frobenius},
unfortunately the general matching given
in~\cite{Clark_Ehrenborg} does not have this property.

Lastly, all of our results are based upon $\Delta$ being a filter
in the composition lattice $\Comp(n)$. What if we remove
the filter constraint?
That is, let $\Omega$ be an arbitrary collection of compositions of $n$ not containing
the extreme composition~$(n)$.
Define~$Q^{*}_{\Omega}$ to be all ordered set partitions
$\sigma = (C_{1}, C_{2}, \ldots, C_{k})$
such that $\type(\sigma) \in \Omega$ and
containing $n$ in the last block $C_{k}$.
Let $\Pi_{\Omega}$ be the image of $Q^{*}_{\Omega}$
under the forgetful map $f$. What can be said
about the homology groups and the homotopy type of
the order complex $\triangle(\Pi_{\Omega})$?
We need to understand the topology
of the links~$\link_{\vec{c}\,}(\Omega)$, even though these
links are not themselves simplicial complexes.

\section*{Acknowledgements}

The authors thank Bert Guillou and Kate Ponto
for their homological guidance and expertise.
They also thank Sheila Sundaram 
and Michelle Wachs for essential references.
The authors thank Margaret Readdy for
her comments on an earlier draft.
Both authors were partially supported by
National Security Agency grant~H98230-13-1-0280.
The first author wishes to thank the 
Princeton University Mathematics Department
where this work began.

\vspace*{2 mm}

\noindent
{\small \em
R.\ Ehrenborg
and
D.\ Hedmark,
Department of Mathematics,
University of Kentucky,
Lexington, KY 40506,
\{{\tt richard.ehrenborg,dustin.hedmark}\}{\tt@uky.edu}.
}


{\small

\begin{thebibliography}{99}
\bibitem{Billera_Myers}
\journal{L.\ J.\ Billera and A.\ N.\ Meyers}
        {Shellability of interval orders}
        {Order}
        {15}{1998}{113--117}

\bibitem{Bjorner_Wachs_non_pure_I}
\journal{A.\ Bj\"orner and M.\ L.\ Wachs}
            {Shellable nonpure complexes and posets. I.}
            {Trans.\ Amer.\ Math.\ Soc.}
            {348}{1996}{1299--1327}

\bibitem{Bjorner_Wachs_Welker}
\journal{A.\ Bj\"orner, M.\ L.\ Wachs and V.\ Welker}
            {Poset fiber theorems}
            {Trans.\ Amer.\ Math.\ Soc.}
            {357}{2004}{1877--1899}

\bibitem{Browdy}
\thesis{A.\ Browdy}
           {The (Co)Homology of Lattices of Partitions
             with Restricted Block Size}
           {University of Miami}
           {1996}

\bibitem{Calderbank_Hanlon_Robinson}
\journal{A.\ R.\ Calderbank, P.\ Hanlon and R.\ W.\ Robinson}
        {Partitions into even and odd block size and
         some unusual characters of the symmetric groups}
        {Proc.\ London Math.\ Soc.\ (3)}
        {53}{1986}{288--320}

\bibitem{Clark_Ehrenborg}
\journal{E.\ Clark and R.\ Ehrenborg}
            {The Frobenius complex}
            {Ann.\ Comb.}
            {16}{2012}{215--232}


\bibitem{Ehrenborg_Jung}
\journal{R.\ Ehrenborg and J.\ Jung}
           {The topology of restricted partition posets}
           {J.\ Algebraic Combin.}
           {37}{2013}{643--666}

\bibitem{Ehrenborg_Readdy_I}
\journal{R.\ Ehrenborg and M.\ Readdy}
        {The M{\"o}bius function of partitions with restricted block sizes}
        {Adv.\ in Appl.\ Math.}
%%        {Advances in Applied Mathematics}
        {39}{2007}{283--292}

\bibitem{Kozlov_book}
\book{D.\ N.\ Kozlov}
         {Combinatorial Algebraic Topology}
         {Springer--Verlag, Berlin}
         {2008}

\bibitem{Kozlov_sparse_sets}
\journal{D.\ N.\ Kozlov}
            {Complexes of directed trees}
            {J.\ Combin.\ Theory Ser.\ A}
            {88}{1999}{112--122}

\bibitem{Miller}
\journal{A.\ R.\ Miller}
           {Reflection arrangements and ribbon representations}
           {European J.\ Combin.}
           {39}{2014}{24--56}


\bibitem{Quillen}
\journal{D.\ Quillen}
        {Homotopy properties of the poset of
         nontrivial $p$-subgroups of a group}
        {Adv.\ Math.}
        {28}{1978}{101--128}

\bibitem{Sagan}
\book{B.\ E.\ Sagan}
         {The symmetric group.
          Representations, combinatorial algorithms,
          and symmetric functions. Second edition}
         {Springer--Verlag, New York}
         {2001}

\bibitem{exponential_structures}
\journal{R.\ P.\ Stanley}
            {Exponential Structures}
            {Studies Appl. Math}
            {59}{1978}{73--82}

\bibitem{EC2}
\book{R.\ P.\ Stanley}
         {Enumerative Combinatorics, Vol 2}
         {Cambridge University Press, Cambridge}
         {1999}

\bibitem{EC1}
\book{R.\ P.\ Stanley}
         {Enumerative Combinatorics, Vol 1, second edition}
         {Cambridge University Press, Cambridge}
         {2012}
         
\bibitem{Sundaram_Hopf}
{\sc S.\ Sundaram,}
{\it Applications of the Hopf trace formula to
     computing homology representations.}
Jerusalem combinatorics~'93, 277--309,
{\it Contemp.\ Math.}, {\bf 178},
Amer.\ Math.\ Soc., Providence, RI, 1994.

\bibitem{Sundaram}
\journal{S.\ Sundaram}
            {On the topology of two partition posets
            with forbidden block sizes}
            {J.\ Pure Appl.\ Algebra}
            {155}{2001}{271--304}

\bibitem{Sylvester}
\thesis{G.\ S.\ Sylvester}
       {Continuous-Spin Ising Ferromagnets}
       {Massachusetts Institute of Technology}
       {1976}

\bibitem{Wachs}
\journal{M.\ L.\ Wachs}
        {A basis for the homology of the $d$-divisible partition lattice}
        {Adv.\ Math.}
        {117}{1996}{294--318}

\bibitem{Wachs_III}
        {\sc M.\ L.\ Wachs,}
         Poset topology: tools and applications.
         Geometric combinatorics, 497--615,
         IAS/Park City Math.\ Ser., 13,
         Amer.\ Math.\ Soc., Providence, RI, 2007.

\end{thebibliography}
}
\end{document}